\documentclass[12pt]{article}
\usepackage{amssymb,amsmath,latexsym,amsthm}
\usepackage{mathrsfs,mathabx}
\usepackage{environ}
\usepackage{url}
\usepackage{enumerate}
\usepackage{ifpdf}
\ifpdf
    \usepackage[pdftex]{hyperref}
\else
    \usepackage[dvips]{hyperref}
\fi

\title{%
    Finite Markov chains coupled to general Markov processes and an application
    to metastability II
}
\author{%
    Thomas G. Kurtz\\
    University of Wisconsin-Madison
    \and
    Jason Swanson\thanks{%
        Supported in part by the VIGRE grant of University of Wisconsin-Madison
        and by NSA grant H98230-09-1-0079.
    }\\
    University of Central Florida
}
\date{January 6, 2021}

\usepackage{color}

\begin{document}

\newtheorem{thm}{Theorem}[section]
\newtheorem{cor}[thm]{Corollary}
\newtheorem{prop}[thm]{Proposition}
\newtheorem{lemma}[thm]{Lemma}
\newtheorem{assum}[thm]{Assumption}
\newtheorem{condition}[thm]{Condition}
\theoremstyle{definition}
\newtheorem{defn}[thm]{Definition}
\newtheorem{expl}[thm]{Example}
\theoremstyle{remark}
\newtheorem{rmk}[thm]{Remark}
\numberwithin{equation}{section}

\def\al{\alpha}
\def\be{\beta}
\def\ga{\gamma}
\def\Ga{\Gamma}
\def\de{\delta}
\def\De{\Delta}
\def\ep{\varepsilon}
\def\eps{\varepsilon}
\def\ze{\zeta}
\def\th{\theta}
\def\ka{\kappa}
\def\la{\lambda}
\def\La{\Lambda}
\def\vpi{\varpi}
\def\si{\sigma}
\def\Si{\Sigma}
\def\ph{\varphi}
\def\om{\omega}
\def\Om{\Omega}

\def\wt{\widetilde}
\def\wh{\widehat}
\def\wch{\widecheck}
\def\ol{\overline}
\def\ds{\displaystyle}

\def\nab{\nabla}
\def\pa{\partial}
\def\To{\Rightarrow}
\def\eqd{\overset{d}{=}}
\def\emp{\emptyset}

\def\pf{\noindent{\bf Proof.} }
\def\qed{\hfill $\Box$}

\providecommand{\flr}[1]{\left\lfloor{#1}\right\rfloor}
\providecommand{\ceil}[1]{\left\lceil{#1}\right\rceil}
\providecommand{\ang}[1]{\left\langle{#1}\right\rangle}


\def\bA{\mathbb{A}}
\def\bB{\mathbb{B}}
\def\bC{\mathbb{C}}
\def\bD{\mathbb{D}}
\def\bE{\mathbb{E}}
\def\bF{\mathbb{F}}
\def\bG{\mathbb{G}}
\def\bH{\mathbb{H}}
\def\bI{\mathbb{I}}
\def\bJ{\mathbb{J}}
\def\bK{\mathbb{K}}
\def\bL{\mathbb{L}}
\def\bM{\mathbb{M}}
\def\bN{\mathbb{N}}
\def\bO{\mathbb{O}}
\def\bP{\mathbb{P}}
\def\bQ{\mathbb{Q}}
\def\bR{\mathbb{R}}
\def\bS{\mathbb{S}}
\def\bT{\mathbb{T}}
\def\bU{\mathbb{U}}
\def\bV{\mathbb{V}}
\def\bW{\mathbb{W}}
\def\bX{\mathbb{X}}
\def\bY{\mathbb{Y}}
\def\bZ{\mathbb{Z}}

\def\bfA{{\bf A}}
\def\bfB{{\bf B}}
\def\bfC{{\bf C}}
\def\bfD{{\bf D}}
\def\bfE{{\bf E}}
\def\bfF{{\bf F}}
\def\bfG{{\bf G}}
\def\bfH{{\bf H}}
\def\bfI{{\bf I}}
\def\bfJ{{\bf J}}
\def\bfK{{\bf K}}
\def\bfL{{\bf L}}
\def\bfM{{\bf M}}
\def\bfN{{\bf N}}
\def\bfO{{\bf O}}
\def\bfP{{\bf P}}
\def\bfQ{{\bf Q}}
\def\bfR{{\bf R}}
\def\bfS{{\bf S}}
\def\bfT{{\bf T}}
\def\bfU{{\bf U}}
\def\bfV{{\bf V}}
\def\bfW{{\bf W}}
\def\bfX{{\bf X}}
\def\bfY{{\bf Y}}
\def\bfZ{{\bf Z}}

\def\cA{\mathcal{A}}
\def\cB{\mathcal{B}}
\def\cC{\mathcal{C}}
\def\cD{\mathcal{D}}
\def\cE{\mathcal{E}}
\def\cF{\mathcal{F}}
\def\cG{\mathcal{G}}
\def\cH{\mathcal{H}}
\def\cI{\mathcal{I}}
\def\cJ{\mathcal{J}}
\def\cK{\mathcal{K}}
\def\cL{\mathcal{L}}
\def\cM{\mathcal{M}}
\def\cN{\mathcal{N}}
\def\cO{\mathcal{O}}
\def\cP{\mathcal{P}}
\def\cQ{\mathcal{Q}}
\def\cR{\mathcal{R}}
\def\cS{\mathcal{S}}
\def\cT{\mathcal{T}}
\def\cU{\mathcal{U}}
\def\cV{\mathcal{V}}
\def\cW{\mathcal{W}}
\def\cX{\mathcal{X}}
\def\cY{\mathcal{Y}}
\def\cZ{\mathcal{Z}}

\def\sA{\mathscr{A}}
\def\sB{\mathscr{B}}
\def\sC{\mathscr{C}}
\def\sD{\mathscr{D}}
\def\sE{\mathscr{E}}
\def\sF{\mathscr{F}}
\def\sG{\mathscr{G}}
\def\sH{\mathscr{H}}
\def\sI{\mathscr{I}}
\def\sJ{\mathscr{J}}
\def\sK{\mathscr{K}}
\def\sL{\mathscr{L}}
\def\sM{\mathscr{M}}
\def\sN{\mathscr{N}}
\def\sO{\mathscr{O}}
\def\sP{\mathscr{P}}
\def\sQ{\mathscr{Q}}
\def\sR{\mathscr{R}}
\def\sS{\mathscr{S}}
\def\sT{\mathscr{T}}
\def\sU{\mathscr{U}}
\def\sV{\mathscr{V}}
\def\sW{\mathscr{W}}
\def\sX{\mathscr{X}}
\def\sY{\mathscr{Y}}
\def\sZ{\mathscr{Z}}

\def\fA{\mathfrak{A}}
\def\fB{\mathfrak{B}}
\def\fC{\mathfrak{C}}
\def\fD{\mathfrak{D}}
\def\fE{\mathfrak{E}}
\def\fF{\mathfrak{F}}
\def\fG{\mathfrak{G}}
\def\fH{\mathfrak{H}}
\def\fI{\mathfrak{I}}
\def\fJ{\mathfrak{J}}
\def\fK{\mathfrak{K}}
\def\fL{\mathfrak{L}}
\def\fM{\mathfrak{M}}
\def\fN{\mathfrak{N}}
\def\fO{\mathfrak{O}}
\def\fP{\mathfrak{P}}
\def\fQ{\mathfrak{Q}}
\def\fR{\mathfrak{R}}
\def\fS{\mathfrak{S}}
\def\fT{\mathfrak{T}}
\def\fU{\mathfrak{U}}
\def\fV{\mathfrak{V}}
\def\fW{\mathfrak{W}}
\def\fX{\mathfrak{X}}
\def\fY{\mathfrak{Y}}
\def\fZ{\mathfrak{Z}}

\maketitle

\begin{abstract}
    We consider a diffusion given by a small noise perturbation of a dynamical
    system driven by a potential function with a finite number of local minima.
    The  classical results of Freidlin and Wentzell show that the time this
    diffusion spends in the domain of attraction of one of these local minima is
    approximately exponentially distributed and hence the diffusion should
    behave approximately like a Markov chain on the local minima. By the work of
    Bovier and collaborators, the local minima can be associated with the small
    eigenvalues of the diffusion generator. In Part I of this work
    \cite{Kurtz2020}, by applying a Markov mapping theorem, we used the
    eigenfunctions of the generator to couple this diffusion to a Markov chain
    whose generator has eigenvalues equal to the eigenvalues of the diffusion
    generator that are associated with the local minima and established explicit
    formulas for conditional probabilities associated with this coupling. The
    fundamental question now becomes to relate the coupled Markov chain to the
    approximate Markov chain suggested by the results of Freidlin and Wentzel.
    In this paper, we take up this question and provide a complete analysis of
    this relationship in the special case of a double-well potential in one
    dimension.

    \bigskip

    \noindent {\bf AMS subject classifications:} Primary 60J60; secondary 60H10,
    60F10, 60J27, 60J28, 34L10

    \noindent {\bf Keywords and phrases:} conditional distributions, coupling,
    eigenfunctions, Freidlin and Wentzell, Markov mapping theorem, Markov
    processes, metastability
\end{abstract}

\section{Introduction}\label{S:intro}

In the interest of self-containment, we will first recap the essential
definitions from Part I of this work \cite{Kurtz2020}. Fix $\ep>0$ and consider
the stochastic process,
\begin{equation}\label{SDE}
    X_{\ep}(t) = X_{\ep}(0)
        - \int_0^t \nab F(X_{\ep}(s)) \,ds + \sqrt {2\ep}\,W(t),
\end{equation}
where $F\in C^3(\bR^d)$ and $W$ is a standard $d$-dimensional Brownian motion.
Let $\ph$ be the solution to the differential equation $\ph'=-\nab F(\ph)$. We
will use $\varphi_x$ to denote the solution with  $\ph_x(0)=x$. The process
$X_{\ep}$ is a small-noise perturbation of the deterministic process $\ph$.

Suppose $F\in C^3(\bR^d)$ and $\lim_{|x|\to\infty}F(x)=\infty$, and that $\cM=
\{x_0,\ldots ,x_m\}$ is the set of local minima of the $F$, with $m\ge 1$. The
points $x_j$ are stable points for the process $\ph$. For $X_{\ep}$, however,
they are not stable. The process $X_{\ep}$ will initially gravitate toward one
of the $x_j$ and move about randomly in a small neighborhood of this point. But
after an exponential amount of time, a large fluctuation of the noise term will
move the process $X_{\ep}$ out of the domain of attraction of $x_j$ and into the
domain of attraction of one of the other minima. We say that each point $x_j$ is
a point of \textit{metastability} for the process $X_{\ep}$.

If $X$ is a cadlag process in a complete, separable metric space $S$ adapted to
a right continuous filtration (assumptions that are immediately satisfied for
all processes considered here) and $H$ is either open or closed, then
$\tau^X_H=\inf\{t>0:X(t)\text{ or }X(t-)\in H\}$ is a stopping time (see, for
example, \cite[Proposition 1.5]{EK}). If $x\in S$, let
$\tau^X_x=\tau^X_{\{x\}}$.  We may sometimes also write $\tau^X(H)$, and if the
process is understood, we may omit the superscript.

Let
\begin{equation}\label{domatt}
    D_j = \{x \in \bR^d: \lim_{t\to\infty}\ph_x(t) = x_j\}
\end{equation}
be the domains of attraction of the local minima. It is well-known (see, for
example, \cite{FW}, \cite[Theorem 3.2]{BEGK}, \cite[Theorems 1.2 and 1.4]{BGK},
and \cite{E}) that as $\ep \to 0$, $\tau^{X_{\ep}}(D_j^c)$ is asymptotically
exponentially distributed under $P^{x_j}$.  It is therefore common to
approximate the process $X_{\ep}$ by a continuous time Markov chain on the set
$\cM$ (or equivalently on $\{0,\ldots ,m\}$).

In this project, for each $\ep>0$, we wish to capture this approximate Markov
chain behavior by coupling $X_{\ep}$ to a continuous time Markov chain,
$Y_{\ep}$, on $\{0,\ldots ,m\}$. We refer to the indexed collection of
coupled processes, $\{(X_{\ep},Y_{\ep}):\ep>0\}$ as a \textit{coupling
sequence}.

In \cite{Kurtz2020}, we developed a general coupling procedure that goes
beyond the specific case of interest here. It is a construction that builds a
coupling between a Markov process on a complete and separable metric space and a
continuous-time Markov chain where the generators of the two processes have
common eigenvalues. The coupling is done in such a way that observations of the
chain yield quantifiable conditional probabilities about the process.

We then applied this construction to the special case of a reversible diffusion
on $\bR^d$ driven by a potential function and a small white noise perturbation.
We summarize here the results in this special case. Assume there exist constants
$a_i>0$ and $c_i>0$ such that $a_2<2a_1-2$, and  
\begin{align}
    c_1|x|^{a_1} - c_2 &\le |\nab F(x)|^2 \le c_3|x|^{a_2}+c_4,
        \label{superquad1}\\
    c_1|x|^{a_1} - c_2
        &\le (|\nab F(x)| - 2\De F(x))^2 \le c_3|x|^{a_2}+c_4.
        \label{superquad2}
\end{align}
Let
\[
    A = \{(f, -\ep\wt Hf): f \in C_c^\infty(\bR^d)\}
\]
be the generator for \eqref{SDE}, and let $(-\la_0,\eta_0),\ldots,
(-\la_m,\eta_m)$ be the first $m+1$ eigenvalues and eigenfunctions of $A$.
By \cite[Proposition 3.7]{Kurtz2020}, the functions $\eta_k$ are continuous and
bounded. We may therefore choose a matrix, $Q\in\bR^{(m+1)\times(m+1)}$, and
vectors, $\xi^{(1)},\ldots,\xi^{(m)}$, such that
\begin{enumerate}[(i)]
    \item $Q$ is the generator of a continuous-time Markov chain with state
    space $E_0=\{0, 1, \ldots, m\}$,

    \item $\xi^{(k)}$ is a right eigenvector of $Q$ with eigenvalue $-\la_k$,
    and

    \item for $0\le i\le m$, the functions,
    \[
        \al_i(x) = 1 + \sum_{k=1}^m\xi_i^{(k)}\eta_k(x),
    \]
    are strictly positive.
\end{enumerate}
We then choose a probability measure, $p=(p_0,\ldots,p_m)$, on $E_0$, define the
measure $\nu$ on $\bR^d\times E_0$ by
\begin{equation}\label{nudef}
    \nu(\Ga \times \{i\})
        = p_i\al(i,\Ga), \quad \Ga \in \cB(\bR^d),
\end{equation}
and let $(X_\ep,Y_\ep)$ be the cadlag Markov process on $\bR^d\times E_0$ with
initial
distribution $\nu$ and generator,
\begin{equation}\label{Bgen}
    Bf(x,i) = Af(x,i)
        + \sum_{j\ne i}Q_{ij}\frac {\al_j(x)}{\al_i(x)}(x)(f(x,j) - f(x,i)).
\end{equation}
Note that all of these objects ($A$, $\la_k$, $\eta_k$, $Q$, $\xi^{(k)}$, $p$,
and so on) depend on $\ep$, though this dependence is suppressed in the notation
for readability.

By \cite[Theorem 3.8]{Kurtz2020}, the process $X_\ep$ solves \eqref{SDE}, the
process $Y_\ep$ has generator $Q$, and
\begin{equation}\label{mainRd}
    P(X(t) \in \Ga \mid Y(t) = j) = \int_\Ga \al_j(x)\,\vpi(dx),
\end{equation}
for all $t\ge0$, all $0\le j\le m$, and all $\Ga\in\cB(E)$.

In this way, for each $\ep>0$, we create a coupling, $(X_\ep,Y_\ep)$. We
referred to the indexed collection of coupled processes, $\{(X_{\ep},Y_
{\ep}):\ep>0\}$, as a \textit{coupling sequence}. Our objective is to
investigate the possibility of constructing a coupling sequence which satisfies
both
\begin{equation}\label{tracking-req}
    P(X_{\ep}(t) \in D_j \mid Y_{\ep}(t) = j) \to 1
\end{equation}
and
\begin{equation}\label{rate-match-req}
    E^i[\tau_j^{Y_\ep}] \sim E^{x_i}[\tau^{X_\ep}_{B_{\rho}(x_0)}]
\end{equation}
as $\ep\to0$, for all $i$ and $j$, where $B_{\rho}(x)$ is the ball of radius
$\rho$ centered at $x$.

In the current paper, we consider this question in the case of a double-well
potential in one dimension. That is, suppose $d=1$ and $\cM=\{x_0,x_1\}$, where
$x_0<0<x_1$. Let $F$ be decreasing on $(-\infty ,x_0)$ and $(0,x_1)$, and
increasing on $(x_0,0)$ and $(x_1,\infty )$, and satisfy $F(x_0)<F(x_1)$. Then
the domains of attraction are $D_0=(-\infty ,0)$ and $D_1=(0,\infty )$. There
are many possible coupling sequences, so for each such sequence, we can ask if
it satisfies any of the following:
\begin{align}
    &P(X_\ep(t) < 0 \mid Y_\ep(t) = 0) \to 1,
        \label{trackdeep}\\
    &P(X_\ep(t) > 0 \mid Y_\ep(t) = 1) \to 1,
        \label{trackshallow}\\
    &E^1[\tau^{Y_\ep}_0] \sim E^{x_1}[\tau^{X_\ep}_{B_{\rho}(x_0)}],
        \label{time2deep}\\
    &E^0[\tau^{Y_\ep}_1] \sim E^{x_0}[\tau^{X_\ep}_{B_{\rho}(x_1)}]
        \label{time2shallow},
\end{align}
as $\ep\to 0$, where $0<\rho <|x_0|\wedge x_1$.

Let $-\la_\ep$ be the second eigenvalue of the generator of $X_\ep$. It is known
(see, for example, \cite{Sugiura1995,Sugiura2001} or  \cite{BEGK,BGK}), that in
\eqref{evasym} and \eqref{ttasym}, we have
\begin{align*}
    E^{x_1}[\tau^{X_\ep}_{B_\rho(x_0)}]
        &\sim \frac {2\pi}{|F''(0)F''(x_1)|^{1/2}} e^{(F(0)-F(x_1))/\ep}
        \sim \frac 1{\la_\ep},\\
    E^{x_0}[\tau^{X_\ep}_{B_\rho(x_1)}]
        &\sim \frac {2\pi}{|F''(0)F''(x_0)|^{1/2}} e^{(F(0)-F(x_0))/\ep}.
\end{align*}
Thus, \eqref{time2deep} and \eqref{time2shallow} are equivalent to
\eqref{t-todeep} and \eqref{t-fromdeep}, respectively. Moreover, Theorem
\ref{T:halfmain} shows that, in our coupling construction, \eqref{trackdeep} is
equivalent to the assertion that, given $Y(t)=0$, the distribution of $X(t)$ is
asymptotically equivalent to the stationary distribution, conditioned to be on
$(-\infty ,0)$. Theorem \ref{T:halfmain2} gives the analogous equivalency for
\eqref{trackshallow}.

In Section \ref{S:asympcoup}, we will show that, in our coupling construction,
\eqref{trackshallow} implies \eqref{time2deep}, which implies \eqref{trackdeep},
and \eqref{time2shallow} implies \eqref{time2deep}. We also show by example that
there are no other implications among these conditions. For example, we can
couple $X_\ep$ and $Y_\ep$ so that \eqref{trackdeep}, \eqref{time2deep}, and
\eqref{time2shallow} are satisfied, but \eqref{trackshallow} is not. In other
words, it is possible to build the Markov chain with asymptotically the same
transition rates as the process, but the two do not remain synchronized, in the
sense that \eqref{trackshallow} fails. Or, as another example, we can couple the
processes so that \eqref{trackdeep}-\eqref{time2deep} are satisfied, but
\eqref{time2shallow} is not. In other words, we can have a coupling where the
Markov chain accurately tracks the diffusion, but the transition rates of the
two processes are not the same.

In the case of the double-well potential, for fixed $\ep>0$, the dynamics of the
coupling $(X_\ep,Y_\ep)$ are uniquely determined by two parameters,
$\xi_{1,\ep}$ and $\xi_{2,\ep}$ (see Lemma \ref{L:matcond}). If we identify
coupling sequences whose parameters are asymptotically equivalent as $\ep\to 0$,
then there is a unique coupling sequence satisfying
\eqref{trackdeep}-\eqref{time2shallow}. Heuristically, we build this sequence by
choosing the $\xi$'s so that $\al_j\approx c_{j,\ep}1_{D_j}$. More specifically,
we choose them so that $\al_0=-\eta_1 /\eta_1(\infty)+1$ and
$\al_1=\eta_1/|\eta_1(-\infty)|+1$. We then prove sharp enough bounds on the
behavior of $\eta_1$ to show that the approximation $\al_j\approx
c_{j,\ep}1_{D_j}$ is sufficiently accurate.

The outline of the paper is as follows. In Section \ref{S:order}, we address the
issue of how the minima should be ordered so that they correspond to the
eigenvalues of the generator of the diffusion. This is a necessary prerequisite
for attaining the asymptotic behavior in \eqref{tracking-req} and 
\eqref{rate-match-req}. In Section \ref{S:eigstruc}, we specialize to the case
of the double-well potential in $d=1$. We begin there with the study the
structure of the second eigenfunction. In particular, we narrow down the
location of the nodal point, show that the eigenfunction is asymptotically flat
near the minima, and establish key estimates on the behavior of the
eigenfunction near the saddle point. Then, in Section \ref{S:asympcoup}, we use
these results to give a complete analysis of our coupling sequences for the
double-well potential.

\section{Ordering the local minima}\label{S:order}

Heretofore, no mention has been made of the order in which the local minima,
$\cM = \{x_0,\ldots,x_m\}$, are listed. No particular order is necessary in
order to construct a coupling sequence. But if that sequence is to exhibit the
behavior in \eqref{tracking-req} and \eqref{rate-match-req}, then the minima
should be ordered so that they correspond with the eigenvalues of $A$.

To describe this ordering, we first establish some notation and terminology. For
any two sets $A,B\subset\bR^d$, define the set of paths from $A$ to $B$ as
\[
    \cP^*(A,B) = \{\om \in C([0,1];\bR^d): \om(0) \in A, \, \om(1) \in B\}.
\]
Given $F:\bR^d\to\bR$, the height of the saddle, or communication height,
between $A$ and $B$ is defined as
\[
    \wh F(A,B) = \inf_{\om\in\cP^*(A,B)} \sup_{t\in [0,1]} F(\om(t)).
\]
The set of minimal paths from $A$ to $B$ is
\[
    \cP(A,B) = \{\om\in\cP^*(A,B): \sup_{t\in [0,1]} F(\om(t)) = \wh F(A,B)\}.
\]
A gate, $G(A,B)$, is a minimal subset of $\{z\in\bR^d:F(z)=\wh F(A,B)\}$ such
that all minimal paths intersect $G(A,B)$. In general, $G(A,B)$ is not unique.
The set of saddle points, $\cS(A,B)$, is the union of all gates.

\begin{assum}\label{asmF2}
    \begin{enumerate}[(i)]
        \item   For $x,y\in\cM$, $G(x,y)$ is unique and consists of a finite set
                of isolated points $\{z_i^*(x,y)\}$.
        \item   The Hessian matrix of $F$ is non-degenerate at each $x\in\cM$
                and at each saddle point $z_i^*(x,y)$.
        \item   The minima $\cM=\{x_0,\ldots,x_m\}$ can be labeled in such a way
                that, with $\cM_k=\{x_0,\ldots,x_k\}$, each saddle point
                $z^*(x_k,\cM_{k-1})$ is unique, the Hessian matrix of $F$ at
                $z^*(x_k,\cM_{k-1})$ is non-degenerate, and
                \begin{equation}\label{sadd}
                    \wh F(x_k,\cM_k\setminus x_k) - F(x_k)
                        < \wh F(x_i,\cM_k\setminus x_i) - F(x_i),
                \end{equation}
                for all $0\le i<k\le m$.
    \end{enumerate}
\end{assum}

We shall assume our potential function $F$ satisfies Assumption \ref{asmF2}, and
that the minima are ordered as in (iii).

\section{Structure of the second eigenfunction}\label{S:eigstruc}

\subsection{Tools and preliminary results}

From this point forward, we take $d=1$. Note that
$\ep\eta_k''-F'\eta_k'=-\la_k\eta_k$ for all integers $k\ge 0$. We will make use
of the fact that the eigenfunctions satisfy the integral equations in the
following lemma.

\begin{lemma}\label{L:eigenintrep}
    For any $k\in\bN$,
    \begin{align}
        \eta_k(x) &= \eta_k(\infty)
            - \frac {\la_k}\ep\int_x^\infty\int_x^u
            \exp\left({\frac{F(v)-F(u)}\ep}\right)\eta_k(u)\,dv\,du
            \label{eigenintrep1}\\
        &= \eta_k(-\infty)
            -\frac{\la_k}\ep\int_{-\infty}^x\int_u^x
            \exp\left({\frac{F(v)-F(u)}\ep}\right)\eta_k(u)\,dv\,du.
            \label{eigenintrep2}
    \end{align}
\end{lemma}

\begin{proof}
    Fix $k\in\bN$. Since $\eta_k$ is bounded by Proposition \cite[Proposition
    3.7]{Kurtz2020}, we may choose $C_1>0$ such that $|\eta_k(x)|\le C_1$ for
    all $x\in\bR$. Now fix $x\in\bR$. Since $a_1>2$, we may choose
    $\al\in(1,a_1/2)$. By \cite[Lemma 3.3]{Kurtz2020}, assumptions 
    \eqref{superquad1} and \eqref{superquad2} imply that
    \begin{equation}\label{superquad3}
        \wt c_1|x|^{\wt a_1} - \wt c_2 \le |F(x)|
            \le \wt c_3|x|^{\wt a_2} + \wt c_4,
    \end{equation}
    where $\wt a_i = a_i/2+1$. It follows that 
    $\lim_{u\to\infty}u^{-\al}e^{F(u)/\ep}=\infty$. Also by \eqref{superquad1},
    for $u$ sufficiently large, $|u^{-\al}F'(u)|\ge C|u|^{a_1/2-\al}$ for some
    $C>0$. Hence, by L'H\^optal's rule,
    \[
        \lim_{u\to\infty}\frac{\int_x^u e^{F(v)/\ep}\,dv}{u^{-\al}e^{F(u)/\ep}}
            = \lim_{u\to\infty}\frac 1{-\al u^{-(\al+1)} + u^{-\al}F'(u)} = 0,
    \]
    and so we may choose $C_2>0$ such that $\int_x^u e^{F(v)/\ep}\,dv\le C_2
    u^{-\al}e^{F(u)/\ep}$ for all $u\ge x$. Therefore,
    \[
        \int_x^{\infty}\int_x^u \left|{
                \exp\left({\frac{F(v)-F(u)}\ep}\right)\eta_k(u)
            }\right|\,dv\,du \le C_1C_2\int_x^\infty u^{-\al}\,du < \infty,
    \]
    and so the right-hand side of \eqref{eigenintrep1} is well-defined.  

    Let
    \[
        y(x) = \eta_k(\infty)
            - \frac {\la_k}\ep\int_x^\infty\int_x^u
            \exp\left({\frac{F(v)-F(u)}\ep}\right)\eta_k(u)\,dv\,du.
    \]
    Then
    \[
        y'(x) = \frac{\la_k}\ep\int_x^{\infty}
            \exp\left({\frac{F(x)-F(u)}\ep}\right)\eta_k(u)\,du,
    \]
    and
    \[
        y''(x) = -\frac{\la_k}\ep\eta_k(x)
            + F'(x)\frac{\la_k}{\ep^2}\int_x^\infty
            \exp\left({\frac{F(x)-F(u)}\ep}\right)\eta_k(u)\,du.
    \]
    Thus, $\ep y''-F'y'=-\la_k\eta_k=\ep\eta_k''-F'\eta_k'$, so that $y-\eta_k$
    is an eigenfunction corresponding to $\la_0$. That is, $y$ and $\eta_k$
    differ by a constant. But $y(\infty)=\eta_k(\infty)$, so $y=\eta_k$ and this
    proves \eqref{eigenintrep1}.

    By replacing $F$ with $x\mapsto F(-x)$, equation \eqref{eigenintrep1} gives
    \[
        \eta_k(-x) = \eta_k(-\infty)
            - \frac{\la_k}\ep\int_x^\infty\int_x^u
            \exp\left({\frac{F(-v)-F(-u)}\ep}\right)\eta_k(-u)\,dv\,du,
    \]
    which gives
    \begin{align*}
        \eta_k(x) &= \eta_k(-\infty)
            - \frac{\la_k}\ep\int_{-x}^\infty\int_{-u}^x
            \exp\left({\frac{F(v')-F(-u)}\ep}\right)\eta_k(-u)\,dv'\,du\\
        &= \eta_k(-\infty)
            - \frac{\la_k}\ep\int_{-\infty}^x\int_{u'}^x
            \exp\left({\frac{F(v')-F(u')}\ep}\right)\eta_k(u')\,dv'\,du',
    \end{align*}
    proving \eqref{eigenintrep2}.
\end{proof}

We now assume that for some fixed $\wt x_0<0<\wt x_1$:
\begin{enumerate}[(i)] 
    \item   $F$ is strictly decreasing on $(-\infty ,\wt x_0)$ and $(0,\wt
            x_1)$, and strictly increasing on $(\wt x_0,0)$ and $(\wt
            x_1,\infty)$.
    \item   $F''(\wt x_0)>0$, $F''(0)<0$, $F''(\wt x_1)>0$.
    \item   $F(\wt x_0)\ne F(\wt x_1)$.
\end{enumerate}
Then $\cM=\{\wt x_0,\wt x_1\}$ and $m=1$. If $F(\wt x_0)<F(\wt x_1)$, then
\[
    \wh F(\wt x_1,\{\wt x_0\}) - F(\wt x_1) = F(0) - F(\wt x_1)
        < F(0) - F(\wt x_0) = \wh F(\wt x_0,\{\wt x_1\}) - F(\wt x_0),
\]
which would imply $x_0=\wt x_0$, and $x_1=\wt x_1$. On the other hand, if $F(\wt
x_1)<F(\wt x_0)$, then $x_0=\wt x_1$ and $x_1=\wt x_0$. For now, we will not
assume either ordering of the local minima, so that our assumptions are
symmetric under the reflection $x\mapsto -x$. Because of this, results that are
stated in terms of $\wt x_0$ can be applied to $\wt x_1$ by replacing $F(x)$
with $F(-x)$.

Let $\eta =\eta_1$ and $\la=\la_1$. By Courant's nodal domain theorem
\cite[Section VI.6, p.454]{Courant1953}, replacing $\eta$ by $-\eta$ if
necessary, there exists $r=r_\ep\in\bR$ such that
\[
    \eta (x)\begin{cases}
        < 0 &\text{if $x < r_\ep$},\\
        = 0 &\text{if $x = r_\ep$},\\
        > 0 &\text{if $x > r_\ep$}.
    \end{cases}
\]
It therefore follows from Lemma \ref{L:eigenintrep} that $\eta$ is strictly
increasing.  

By \cite[Theorem 1.2]{BGK},
\begin{equation}\label{evasym}
    \la = \frac{|F''(0)F''(x_1)|^{1/2}}{2\pi}\,e^{-(F(0)-F(x_1))/\ep}
        (1+O(\ep^{1/2}|\log\ep|)).
\end{equation}
By \cite[(3.3)]{BEGK}, we have
\begin{equation}\label{ttasym}
    E^{x_j}[\tau^X_{B_\rho(x_{1-j})}]
        \sim \frac{2\pi}{|F''(0)F''(x_j)|^{1/2}}\,e^{(F(0)-F(x_j))/\ep},
\end{equation}
for $0<\rho <|\wt x_0|\wedge |\wt x_1|$. And the following special case of
\cite[Proposition 3.3]{BGK} gives us a way to estimate the shape of the
eigenfunction.  

\begin{thm}\label{T:dirshape}
    Let $h(y)=P^y(\tau^X_{(x_0-\ep,x_0+\ep)}<\tau^X_{r_\ep})$ and
    $\phi(y)=|\eta(y)|/|\eta(x_0+\ep)|$. Then there exists $C,\al,\ep_0>0$ such
    that
    \[
        h(y) \le \phi(y) \le h(y)(1 + C\ep^{\al/2}),
    \]
    for all $y<r_\ep$ and all $\ep\in(0,\ep_0)$.
\end{thm}

To apply this result, we will use the following two lemmas, which formulate the
Freidlin and Wentzell results in our specific case.  

\begin{lemma}\label{L:escape}
    Let $a<\wt a<\wt x_0<\wt b<b<0$ and fix $\de>0$. Then there exists $\ep_0>0$
    such that
    \[
        \exp\left({\frac{1-\de}\ep(F(a)\wedge F(b) - F(\wt x_0))}\right)
            \le E^x[\tau^X_{(a,b)^c}]
            \le \exp\left({
                    \frac{1+\de}\ep(F(a)\wedge F(b) - F(\wt x_0)
                )}\right),
    \]
    for all $\wt a\le x\le\wt b$ and all $\ep\in(0,\ep_0)$. The analogous result
    also holds when $0<a<\wt a<\wt x_1<\wt b<b$.
\end{lemma}

\begin{proof}
    By Theorem \ref{T:FandW}, $\ep\log E^x[\tau^X_{(a,b)^c}]\to L:=F(a)\wedge
    F(b)-F(\wt x_0)$ as $\ep\to 0$, uniformly in $x$ on $[\wt a,\wt b]$. Thus,
    there exists $\ep_0$ such that $\ep\in (0,\ep_0)$ implies $\ep\log
    E^x[\tau_{(a,b)^c}]\le (1+\de)L$, which gives the upper bound. The lower
    bound is deduced similarly.  
\end{proof}

\begin{lemma}\label{L:whichside}
    Let $a<\wt x_0<b<0$ or $0<a<\wt x_1<b$ and define $G=(a,b)$. Assume $F(a)\ne
    F(b)$ and choose $y\in\{a,b\}$ such that $F(y)=F(a)\vee F(b)$. Then, for all
    compact $K\subset G$ and all $\ga>0$, there exists $\ep_0>0$ such that
    \[
        \exp\left({-\frac{|F(a) - F(b)| + \ga}\ep}\right)
            \le P^x(X(\tau^X_{G^c}) = y)
            \le \exp\left({-\frac{|F(a) - F(b)| - \ga}\ep}\right),
    \]
    for all $x\in K$ and all $\ep\in(0,\ep_0)$.
\end{lemma}

\begin{proof}
    We prove only the case where $a<\wt x_0<b$ and $F(a)>F(b)$, so that $y=a$.
    The proofs of the other cases are similar. We use Theorem \ref{T:FandW2},
    Proposition \ref{P:FandW}, and Lemma \ref{L:FandW}. Note that, according to
    the discussion preceding Theorem \ref{T:FandW2}, we have $V_G(x,y)=V(x,y)$
    for all $x,y\in [a,b]$.  

    Fix $x\in K$. In this case,
    \[
        M_G = V_G(\{\wt x_0\}, \{a,b\}) = V_G(\wt x_0, a) \wedge V_G(\wt x_0, b)
            = 2(F(b) - F(\wt x_0)),
    \]
    and
    \begin{align*}
        M_G(x,a) &= \min\{V_G(\wt x_0,x) + V_G(x,a),
            V_G(\wt x_0,\{a,b\}) + V_G(x,a),
            V_G(x,\wt x_0) + V_G(\wt x_0,a)\}\\
        &= \min\{2(F(x) - F(\wt x_0)) + V_G(x,a),
            2(F(b) - F(\wt x_0)) + V_G(x,a),
            2(F(a) - F(\wt x_0))\}
    \end{align*}
    If $a<x<\wt x_0$, then $V_G(x,a)=2(F(a)-F(x))$, so that
    \begin{align*}
        M_G(x,a) &= 2\min\{F(a) - F(\wt x_0),
            F(b) - F(\wt x_0) + F(a) - F(\wt x_0),
            F(a) - F(\wt x_0)\}\\
        &= 2(F(a) - F(\wt x_0)).
    \end{align*}
    If $\wt x_0\le x<b$, then $V_G(x,a)=2(F(a)-F(\wt x_0))$, so that
    \begin{align*}
        M_G(x,a) &= 2\min\{F(x) + F(a) - 2F(\wt x_0),
            F(b) + F(a) - 2F(\wt x_0),
            F(a) - F(\wt x_0)\}\\
        &= 2(F(a) - F(\wt x_0)).
    \end{align*}
    Thus, $M_G(x,a)-M_G=2(F(a)-F(b))$, and the result follows from Theorem
    \ref{T:FandW2}.
\end{proof}

\subsection{Location of the nodal point}

Our first order of business is to identify an interval in which the nodal point
(that is, the zero of the second eigenfunction) is asymptotically located. The
essential feature of the interval is that it is bounded away from the minima as
$\ep\to 0$.

The statement of this result is Corollary \ref{C:nodeloc}. To prove this result,
we need four lemmas, all concerning stopping times of $X$.

\begin{lemma}\label{L:infrominf}
    There exists $R>0$ such that
    $\sup\{E^x[\tau^X_K]:x\in\bR^d,\ep\in(0,1)\}<\infty$, where $K=\ol{B_R(0)}$.
     
\end{lemma}

\begin{proof}
    In this proof, for $r>0$, let
    $\si_r=\tau^{F(X)}_{(-\infty,r]}=\inf\{t\ge0:F(X(t))\le r\}$.

    Choose $C_1,C_2,L>0$ such that  
    \begin{enumerate}[(i)] 
        \item   $V(x)\ge C_1|x|^{a_1}$,
        \item   $C_1|x|^{a_1}\le |\nab F(x)|^2\le C_2|x|^{a_2}$, and
        \item   $C_1|x|^{\wt a_1}\le F(x)\le C_2|x|^{\wt a_2}$,
    \end{enumerate}
    for all $|x|>L$, where $\wt a_j$ are as in \eqref{superquad3}. Choose $R>L$
    such that
    \[
        I := (1 \vee \sup_{|x| \le L} F(x), C_1 R^{\wt a_1}]
            \cap \bN \neq \emptyset,
    \]
    and choose $b\in I$.  

    Suppose $\om\in \{\tau_K>t\}$. Then, for all $s\le t$, we have that
    $|X(s)|>R>L$, and so it follows that $F(X(s))\ge C_1|X(s)|^{\wt
    a_1}>C_1R^{\wt a_1}\ge b$. Thus, $\om\in \{\si_b>t\}$, and we have shown
    that $\tau_K\le\si_b$ a.s. It therefore suffices to show that $E^x[\si_b]$
    is bounded above by a constant that does not depend on $x$ or $\ep$.

    Fix $\ep\in (0,1)$. Let $r=a_1/\wt a_2$ and $C_3=C_1C_2^{-r}$. We will first
    prove that if $x\in\bR^d$, $n\in\bN$, and $b\le n<F(x)\le n+1$, then 
    \begin{equation}\label{indtool}
        E^x[\si_n] \le 2C_3^{-1}n^{-r}.
    \end{equation}
    Let $x$ and $n$ satisfy the assumptions. Using It\^o's rule, we can write
    \[
        F(X(t)) = F(x) + \sqrt{2\ep}\,M(t) - 2\ep\int_0^t \psi (X(s))\,ds,
            P^x\text{-a.s.}
    \]
    where $M(t)=\int_0^t\nab F(X(s))\,dW(s)$ and $\psi=\ep V+|\nab F|^2/(4\ep)$.
    Let $\wt W(s)=M(T(s))$, where the stopping time $T(s)$ is defined by
    $T(s)=\inf\{t\ge 0:[M]_t>s\}$. By \cite[Theorem 3.4.6]{Karatzas1991}, $\wt
    W$ is a standard Brownian motion, and $M(t)=\wt W([M]_t)$. Moreover, by
    \cite[Problem 3.4.5]{Karatzas1991}, $s<[M]_t$ if and only if $T(s)<t$, and
    $[M]_{T(s)}=s$ for all $s\ge 0$.

    Let
    \[
        \wh W(t) = \wt W(t) - \frac 1{2\sqrt {2\ep}}\,t,
    \]
    and define $\wt\si_n=\tau^{\sqrt {2\ep}\,\wh W}_{(-\infty ,n-F(x)]}=\inf
    \{t\ge 0:\wh W(t)\le (n-F(x))/\sqrt {2\ep}\}$. We will prove that
    $[M]_{\si_n}\le\wt\si_n$ a.s. Note that 
    \begin{align*}
        \{\wt\si_n < [M]_{\si_n}\}
            &= \bigcup_{s\in\bQ}\left({
                \{s < [M]_{\si_n}\} \cap \left\{{
                \wh W(s) \le \frac{n - F(x)}{\sqrt{2\ep}}
                }\right\}
            }\right)\\
        &=  \bigcup_{s\in\bQ}\left({
            \{T(s) < \si_n\} \cap \left\{{
                \wh W([M]_{T(s)}) \le\frac{n - F(x)}{\sqrt{2\ep}}
            }\right\}
        }\right).
    \end{align*}
    On the event $\{T(s)<\si_n\}$, we have, for all $u\le T(s)$,
    \begin{equation}\label{pretaun}
        F(X(u)) > n \ge b > \sup_{|x|\le L}F(x),
    \end{equation}
    where the first inequality comes from the definition of $\si_n$. It follows
    that $|X(u)|>L$. Thus, by (i), we have $V(X(u))>0$, and so
    $\psi(X(u))>|\nab F(X(u))|^2/(4\ep)$. Hence,
    \begin{align*}
        n < F(X(T(s))) &\le F(x) + \sqrt{2\ep}M(T(s))
            - \frac12\int_0^{T(s)} |\nab F(X(u))|^2\,du\\
        &= F(x) + \sqrt{2\ep}\,\wt WB([M]_{T(s)}) - \frac12[M]_{T(s)}\\
        &= F(x) + \sqrt{2\ep}\,\wh W([M]_{T(s)}).
    \end{align*}
    Therefore, $\wh W([M]_{T(s)})>(n-F(x))/\sqrt{2\ep}$ a.s.~on the event
    $\{T(s)<\si_n\}$, which shows that $P(\wt\si_n<[M]_{\si_n})=0$.

    Note that for all $|x|>L$, we have
    \[
        |\nab F(x)|^2 \ge C_1|x|^{a_1} = C_1(|x|^{\wt a_2})^{a_1/\wt a_2}
            \ge C_1(C_2^{-1}F(x))^{a_1/\wt a_2} = C_3F(x)^r.
    \]
    Thus, as in \eqref{pretaun}, we obtain
    \[
        \wt\si_n \ge [M]_{\si_n} = \int_0^{\si_n} |\nab F(X(u))|^2\,du
            \ge C_3\int_0^{\si_n} F(X(u))^r\,du \ge C_3n^r\si_n.
    \]
    Hence, using \cite[Exercise 3.5.10]{Karatzas1991}, which gives the Laplace
    transform of $\wt\si_n$, we have
    \[
        E^x[\si_n] \le C_3^{-1}n^{-r}E^x[\wt\si_n]
            = 2C_3^{-1}n^{-r}(F(x) - n)
            \le 2C_3^{-1}n^{-r},
    \]
    which proves \eqref{indtool}. It now follows by induction and the Markov
    property that
    \[
        E^x[\si_b] \le 2C_3^{-1}\sum_{j=b}^n j^{-r},
    \]
    whenever $b\le n<F(x)\le n+1$. Since
    \[
        \wt a_2 = \frac {a_2}2 + 1 < \frac{2a_1-2}2 + 1 = a_1,
    \]
    it follows that $r>1$. Hence, $C_4:=\sum_{j=b}^\infty j^{-r}<\infty$. Since
    $\si_b=0$, $P^x$-a.s., whenever $F(x)\le b$, we have that $E^x[\si_b] \le
    2C_3^{-1}C_4$ for all $x\in\bR^d$.
\end{proof}

\begin{lemma}\label{L:intomin}
    Let $x<\wt x_0$. Then there exists $\ep_0>0$ such that
    \[
        \sup\{E^y[\tau^X_x]: y < x, \ep \in (0,\ep_0)\} < \infty .
    \]
\end{lemma}

\begin{proof}
    Choose $R>|x|$ as in Lemma \ref{L:infrominf}, so that there exists $C_1>0$
    such that $E^y[\tau^X_{-R}]\le C_1$ for all $y<-R$ and all $\ep\in (0,1)$.  

    Suppose $-R<x<\wt x_0$ and $\ep\in (0,1)$. Let $J=(-R-1,x)$. Since
    $\tau^X_{J^c}\le\tau^X_x$ $P^{-R}$-a.s., the strong Markov property gives
    \[
        E^{-R}[\tau^X_x]
            = E^{-R}[\tau^X_{J^c}] + E^{-R}[E^{X(\tau^X_{J^c})}[\tau^X_x]]
            = E^{-R}[\tau^X_{J^c}] + p_{\ep}E^{-R-1}[\tau^X_x],
    \]
    where $p_{\ep}=P^{-R}(X(\tau^X_{J^c})=-R-1)$. Also by the strong Markov
    property and Lemma \ref{L:infrominf},
    \[
        E^{-R-1}[\tau^X_x] = E^{-R-1}[\tau^X_{-R}] + E^{-R}[\tau^X_x]
            \le C_1 + E^{-R}[\tau^X_x].
    \]
    Thus,
    \[
        E^{-R}[\tau^X_x]
            \le \frac{E^{-R}[\tau^X_{J^c}] + p_{\ep}C_1}{1-p_{\ep}}.
    \]
    By Theorem \ref{T:cantdance}, there exists $C_2>0$, $T>0$, and $\ep_0\in
    (0,1)$ such that for all $\ep\in (0,\ep_0)$,
    \[
        E^{-R}[\tau^X_{J^c}] = \int_0^{\infty} P(\tau^X_{J^c}>t)\,dt\le T
            + \int_T^{\infty} e^{-\ep^{-2}C_2(t-T)}\,dt
            \le T + \frac {\ep_0^2}{C_2} =: C_3.
    \]
    Choose $0<r<|\wt x_0|$ such that $F(\wt x_0+r)<F(-R-1)$, and choose
    $\ga<F(-R-1)-F(\wt x_0+r)$.  By Lemma \ref{L:whichside}, making $\ep_0$
    smaller, if necessary, we have
    \[
        p_{\ep} \le P^{-R}(X(\tau^X_{(-R-1,\wt x_0+r)^c}) = -R - 1) \le 
            \exp\left({-\frac{F(-R - 1) - F(\wt x_0 + r) - \ga}\ep}\right),
    \]
    for all $\ep\in(0,\ep_0)$. By making $\ep_0$ even smaller, if necessary, we
    have $p_{\ep}<1/2$ for all $\ep\in(0,\ep_0)$. Thus,
    \[
        E^{-R}[\tau^X_x] \le 2C_3 + C_1 =: C_4,
    \]
    for all $\ep\in(0,\ep_0)$.

    Now, if $y<-R<x<x_0$, then
    \[
        E^y[\tau^X_x] = E^y[\tau^X_{-R}] + E^{-R}[\tau^X_x] \le C_1 + C_4,
    \]
    for all $\ep\in(0,\ep_0)$, and if $-R\le y<x<x_0$, then
    \[
        C_4 \ge E^{-R}[\tau^X_x] = E^{-R}[\tau^X_y] + E^y[\tau^X_x]
            \ge E^y[\tau^X_x],
    \]
    for all $\ep\in(0,\ep_0)$.
\end{proof}

\begin{lemma}\label{L:pastmin}
    For all $\wt x_0<x<0$ and all $\de>0$, there exists $C>0$ and $\ep_0>0$ such
    that for all $0<\ep<\ep_0$ and all $y<x$, we have
    \[
        E^y[\tau^X_x]
            \le C\exp\left({\frac {1 + \de}\ep(F(x) - F(\wt x_0))}\right).
    \]
\end{lemma}

\begin{proof}
    Suppose $\wt x_0<x<0$ and fix $\de>0$. Choose $R>|\wt x_0|$ as in Lemma
    \ref{L:infrominf}, so that there exists $C_1>0$ such that
    $E^y[\tau^X_{-R}]\le C_1$ for all $y<-R$ and all $\ep\in (0,1)$. By making
    $R$ larger, if necessary, we may assume $F(x)<F(-R-1)$. Let $J:=(-R-1,x)$.
    As in the proof of Lemma \ref{L:intomin},
    \[
        E^{-R}[\tau^X_x]
            \le \frac{E^{-R}[\tau^X_{J^c}] + p_{\ep}C_1}{1-p_{\ep}},
    \]
    where $p_{\ep}=P^{-R}(X(\tau^X_{J^c})=-R-1)$. Using Lemma \ref{L:whichside},
    we may choose $\ep_0>0$ such that $p_{\ep}\le 1/2$ for all
    $\ep\in(0,\ep_0)$, giving
    \[
        E^{-R}[\tau^X_x] \le 2E^{-R}[\tau^X_{J^c}] + C_1.
    \]
    As in the proof of Lemma \ref{L:intomin}, if $y<-R$, then
    \[
        E^y[\tau^X_x] = E^y[\tau^X_{-R}] + E^{-R}[\tau^X_x]
            \le E^{-R}[\tau^X_x] + C_1,
    \]
    and if $-R\le y$, then
    \[
        E^y[\tau^X_x] \le E^{-R}[\tau^X_y] + E^y[\tau^X_x]
            = E^{-R}[\tau^X_x] \le E^{-R}[\tau^X_x]+C_1.
    \]
    Thus,
    \[
        E^y[\tau^X_x] \le 2E^{-R}[\tau^X_{J^c}]+2C_1,
    \]
    for all $y<x$ and all $\ep\in(0,\ep_0)$.

    By Lemma \ref{L:escape}, making $\ep_0$ smaller if necessary, we have
    \[
        E^{-R}[\tau^X_{J^c}]
            \le \exp\left({\frac {1+\de}{2\ep}(F(x)-F(\wt x_0)})\right),
    \]
    for all $\ep\in(0,\ep_0)$, which proves the lemma with $C=2+2C_1$.
\end{proof}

\begin{lemma}\label{L:quasistat}
    Let $\vpi_{\eta}(dx)=|\eta (x)|1_{(-\infty ,r_{\ep})}(x)\vpi(dx)$ and
    $\wh\vpi=\vpi_{\eta}((-\infty,r_{\ep}))^{-1}\vpi_{\eta}$. It then follows
    that $P^{\wh\vpi}(\tau^X_{r_{\ep}}>t)=e^{-\la t}$ for all $t\ge 0$.
\end{lemma}

\begin{proof}
    Let $I=(-\infty,r_{\ep})$. Let $X^I$ denote $X$ killed upon leaving $I$.
    Note that $X^I$ with $X^I(0)=x$ solves the martingale problem for
    $(A^I,\de_x)$, where $A^I=\{(f,Af):f\in C_c^{\infty}(\bR),f(r)=0\}$. Choose
    $\ph_n\in C_c^{\infty}(\bR)$ such that $0\le\ph_n\le 1$, $\ph_n(r)=0$, and
    $\ph_n\to 1_I$ pointwise. Then
    \[
        P^{\wh\vpi}(\tau_r>t) = P^{\wh\vpi}(X^I(t) \in I)
            = E^{\wh\vpi}[1_I(X^I(t))]
            = \lim_{n\to\infty}h_n(t),
    \]
    where $h_n(t)=E^{\wh\vpi}[\ph_n(X^I(t))]$. Let $P^I_tf(x)=E^x[f(X^I(t))]$.
    Fix $t\ge 0$ and let $\psi_n=P^I_t\ph_n$. Then
    \[
        h_n(t) = \int_I\psi_n\,d\wh\vpi
            = -\frac 1{\vpi_{\eta}(I)}\int_I\psi_n\eta\,d\vpi,
    \]
    so that
    \begin{multline*}
        h_n'(t) = -\frac 1{\vpi_{\eta}(I)}\int_I (A^I\psi_n)\eta\,d\vpi
            = -\frac 1{\vpi_{\eta}(I)}
                \int_I (\ep\psi_n'' - F'\psi_n')\eta\,d\vpi\\
            = -\frac 1{\vpi_{\eta}(I)}\int_I \psi_n(\ep\eta'' - F'\eta')\,d\vpi
            = \frac\la{\vpi_{\eta}(I)}\int_I \psi_n\eta\,d\vpi
            = -\la h_n(t).
    \end{multline*}
    Thus, $h_n(t)=h_n(0)e^{-\la t}$. Note that $h_n(0)=\int_I\ph_n\,d\wh\vpi \to
    \wh\vpi(I)=1$ as $n\to\infty$. It therefore follows that $P^{\wh\vpi}
    (\tau_r > t)=e^{-\la t}$.
\end{proof}

\begin{thm}\label{T:nodeloc}
    Let $x\in (\wt x_0,0)$ satisfy $F(x)-F(\wt x_0)<F(0)-F(x_1)$. Then there
    exists $\ep_0>0$ such that for all $0<\ep<\ep_0$, we have $x<r_{\ep}$.
\end{thm}

\begin{proof}
    Choose $\de>0$ such that $(1+\de)(F(x)-F(\wt x_0))<F(0)-F(x_1)$. By Lemma
    \ref{L:pastmin}, there exists $\ep_0>0$ and $C_1>0$ such that
    \[
        E^y[\tau^X_x]
            \le C_1\exp\left({\frac{1 + \de}\ep(F(x)-F(\wt x_0)})\right),
    \]
    for all $\ep\in(0,\ep_0)$ and all $y<x$. By \eqref{evasym}, there exists a
    constant $C_2>0$, not depending on $\ep$, such that $\la\le
    C_2e^{-(F(0)-F(x_1))/\ep}$. By making $\ep_0$ smaller if necessary, we may
    assume
    \[
        \ep\log(C_1C_2) < F(0) - F(x_1) - (1 + \de)(F(x) - F(\wt x_0)),
    \]
    for all $\ep\in(0,\ep_0)$.

    Fix $\ep<\ep_0$. Suppose $r_{\ep}\le x$. By Lemma \ref{L:quasistat},
    \begin{multline*}
        C_2^{-1}\exp\left({\frac1\ep(F(0) - F(x_1))}\right) \le \la^{-1}
            = E^{\wh\vpi}[\tau^X_{r_\ep}]
            = \int_{-\infty}^{r_\ep}E^y[\tau^X_{r_\ep}]\wh\vpi(dy)\\
        \le \int_{-\infty}^{r_\ep}E^y[\tau^X_x]\wh\vpi(dy)
            \le \sup_{y<r_{\ep}}E^y[\tau^X_x]
            \le \sup_{y<x}E^y[\tau^X_x]
            \le C_1\exp\left({\frac{1 + \de}\ep(F(x) - F(\wt x_0))}\right),
    \end{multline*}
    which implies
    \[
        \exp\left({
                \frac{F(0) - F(x_1) - (1 + \de)(F(x) - F(\wt x_0))}\ep
            }\right) \le C_1C_2,
    \]
    a contradiction.
\end{proof}

\begin{cor}\label{C:nodeloc}
    Suppose $F(\wt x_0)<F(\wt x_1)$, so that $x_0=\wt x_0$ and $x_1=\wt x_1$.
    Choose $\xi\in(x_0,0)$ such that $F(\xi)-F(x_0)=F(0)-F(x_1)$. Then for all
    $\de>0$, there exists $\ep_0>0$ such that $r_\ep\in(\xi -\de,\de)$ for all
    $0<\ep<\ep_0$.
\end{cor}

\begin{proof}
    Without loss of generality, we may assume $\xi -\de>x_0$ and $\de<x_1$.
    Taking $x=\xi -\de$ in Theorem \ref{T:nodeloc}, we may choose $\ep_1$ such
    that $\xi -\de<r_\ep$ for all $\ep<\ep_1$. For the upper bound on $r_\ep$,
    we apply Theorem \ref{T:nodeloc} to $x\mapsto F(-x)$. In this case, the
    theorem says that if $x\in(-x_1,0)$ satisfies $F(-x)-F(x_1)<F(0)-F(x_0)$,
    then there exists $\ep_2>0$ such that $x<\wt r_\ep$ for all $\ep<\ep_2$,
    where $\wt r_\ep$ is the nodal point of $x\mapsto-\eta(-x)$, that is, $\wt
    r_\ep=-r_\ep$. Taking $x=-\de$ and $\ep_0=\ep_1\wedge\ep_2$ finishes the
    proof.
\end{proof}

\subsection{Behavior near the minima}

Corollary \ref{C:nodeloc} divides the domain of the second eigenfunction,
$\eta$, into three intervals:  two infinite half-lines that each contain one of
the two minima, and a bounded interval separating the half-lines that contains
the nodal point. Our next order of business is to show that $\eta$ is
asymptotically flat on the infinite half-lines. Theorem \ref{T:stepest} gives
this result for the half-line containing $\wt x_0$. Applying Theorem
\ref{T:stepest} to $x\mapsto F(-x)$ gives the result for the half-line
containing $\wt x_1$.

We begin with a lemma. Recall $a_j,\wt a_j$ and $c_j,\wt c_j$ from 
\eqref{superquad1}, \eqref{superquad2}, and \eqref{superquad3}. In
applying this lemma, note that
\[
    \frac{\wt a_2}{\wt a_1}
        = \frac{a_2 + 2}{a_1 + 2}
        < \frac{2a_1}{a_1 + 2}
        < \frac {a_1}2,
\]
where the first inequality comes from $a_2<2a_1-2$ and the second from $a_1>2$.  

\begin{lemma}\label{L:Ftail}
    Let $x\in(\wt x_0,0)$. Suppose $p$ satisfies
    \[
        \frac2{a_1} < p < \frac{\wt a_1}{\wt a_2} \le 1.
    \]
    Then there exists $u_0<-1$ and $C>0$ such that
    \[
        e^{-F(u)/\ep}\int_u^x e^{F(v)/\ep}\,dv \le C\ep|u|^{-pa_1/2},
    \]
    for all $u<u_0$ and all $\ep>0$.
\end{lemma}

\begin{proof}
    Choose $t<\wt x_0$ such that $F(t)=F(x)$. Using \eqref{superquad1}, we may
    choose $u_0<-1$ and $C'>0$ such that
    \begin{enumerate}[(i)] 
        \item   $-|u_0|^p<t$,\label{item1}
        \item   $F(\th)>0$ and $|F'(\th)|\ge C'|\th|^{a_1/2}$, for all
                $\th<-|u_0|^p$, and\label{item2}
        \item   $\ds{\wt c_3|u|^{p\wt a_2-\wt a_1}<\frac{\wt c_1}2}$ and $
                \ds{\wt c_4-\wt c_2\le\frac{\wt c_1}4|u|^{\wt a_1}}$, for all
                $u<u_0$.\label{item3}
    \end{enumerate}
    Let $G(u)=\int_u^xe^{F(v)/\ep}\,dv$ and $H(u)=e^{F(u)/\ep}$. Fix $u<u_0$ and
    let $v=-|u|^p<-|u_0|^p$.  Note that $u<v$.

    By Cauchy's generalized law of the mean,
    \[
        \frac{G(u) - G(v)}{H(u) - H(v)} = \frac{G'(\th)}{H'(\th)},
    \]
    for some $u<\th<v$. From this, we get
    \begin{align*}
        \frac{G(u)}{H(u)} &= \frac {G(v)}{H(u)} + \frac {G'(\th)}{H'(\th)}
            \left({1 - \frac {H(v)}{H(u)}}\right)\\
        &= \frac{G(v)}{H(u)} + \frac\ep{|F'(\th)|}
            \left({1 - \frac {H(v)}{H(u)}}\right)\\
        &\le \frac{G(v)}{H(u)} + \frac\ep{|F'(\th)|}.
    \end{align*}
    By \eqref{item2},
    \[
        \frac\ep{|F'(\th)|} \le \frac\ep{C'|\th|^{a_1/2}}
            \le \frac\ep{C'|v|^{a_1/2}}
            = \frac\ep{C'|u|^{pa_1/2}}.
    \]
    It therefore suffices to show that
    \begin{equation}\label{laststep}
        \frac{G(v)}{H(u)} \le C''\ep|u|^{-pa_1/2},
    \end{equation}
    for some constant $C''$ that does not depend on $u$ or $\ep$.

    By \eqref{superquad3},
    \begin{align*}
        F(v) - F(u) &\le \wt c_3|v|^{\wt a_2} + \wt c_4
                - \wt c_1|u|^{\wt a_1} - \wt c_2\\
        &= (\wt c_3|u|^{p\wt a_2-\wt a_1} - \wt c_1)|u|^{\wt a_1}
                + \wt c_4 - \wt c_2
            \le -\frac{\wt c_1}4 |u|^{\wt a_1},
    \end{align*}
    where the last inequality comes from \eqref{item3}. By \eqref{item1}, we
    have $v<t$, so that $F(v)>F(w)$ for all $w\in (v,x)$. Hence,
    \begin{align*}
        \frac{G(v)}{H(u)} = \int_v^x e^{(F(w)-F(u))/\ep}\,dw
            &\le |v|e^{(F(v)-F(u))/\ep}\\
        &\le |u|^p\exp\left({-\frac{\wt c_1}{4\ep}|u|^{\wt a_1}}\right)\\
        &= \left({\ep|u|^{-pa_1/2}}\right) \frac1\ep |u|^{p\wt a_1}
            \exp\left({-\frac{\wt c_1}{4\ep}|u|^{\wt a_1}}\right).
    \end{align*}
    Since $x\mapsto x^pe^{-\wt c_1x/4}$ is bounded on $[0,\infty)$, this proves
    \eqref{laststep}.
\end{proof}

\begin{thm}\label{T:stepest} 
    Let $x\in(\wt x_0,0)$ satisfy $F(x)-F(\wt x_0)<F(0)-F(x_1)$. Then there
    exists $C>0$ and $\ep_0>0$ such that for all $0<\ep<\ep_0$,
    \begin{equation}\label{stepest}
        \left|{1 - \frac{\eta(x)}{\eta(-\infty)}}\right|
            \le \frac C\ep\exp\left({
                -\frac1\ep(F(0) - F(x_1) - F(x) + F(\wt x_0))
            }\right).
    \end{equation}
\end{thm}

\begin{proof}
    Again by \eqref{evasym}, there exists a constant $C_1>0$, not depending on
    $\ep$, such that $\la\le C_1e^{-(F(0)-F(x_1))/\ep}$.

    Let $\ep_0$ be as in Theorem \ref{T:nodeloc}, and let $\ep\in(0,\ep_0)$.
    Choose $t<\wt x_0$ such that $F(t)=F(x)$. By Theorem \ref{T:nodeloc},
    $x<r_\ep$. Since $\eta$ is increasing, $\eta(u)<0$ for all $u\le x$.
    Therefore, by \eqref{eigenintrep2},
    \begin{align*}
        0 < \eta(x) - \eta(-\infty) &= \frac\la\ep\int_{-\infty}^x\int_u^x
            e^{(F(v)-F(u))/\ep}|\eta(u)|\,dv\,du\\
        &\le \frac\la\ep|\eta(-\infty)|\int_{-\infty}^x\int_u^x
            e^{(F(v)-F(u))/\ep}\,dv\,du.
    \end{align*}
    Thus,
    \begin{align}
        \left|{1 - \frac{\eta(x)}{\eta(-\infty)}}\right|
            &\le \frac\la\ep\int_{-\infty}^x\int_u^x e^{(F(v)-F(u))/\ep}\,dv\,du
            \notag\\
        &\le \frac{C_1}\ep\,e^{-(F(0)-F(x_1))/\ep}\int_{-\infty}^x\int_u^x
            e^{(F(v)-F(u))/\ep}\,dv\,du.\label{stepestmid}
    \end{align}
    Choose $p$ as in Lemma \ref{L:Ftail}. Then there exist $u_0<0$ and $C_2>0$
    such that
    \[
        \int_{-\infty}^{u_0}\int_u^x e^{(F(v)-F(u))/\ep}\,dv\,du \le C_3\ep,
    \]
    where $C_3=C_2|u_0|^{1-pa_1/2}/(pa_1/2-1)$. By the proof of Lemma
    \ref{L:Ftail}, we have $u_0<t$, and so
    \[
        \int_{u_0}^t\int_u^x e^{(F(v)-F(u))/\ep}\,dv\,du
            \le \int_{u_0}^t (x - u)\,du\le |u_0|^2.
    \]
    Lastly,
    \[
        \int_t^x\int_u^x e^{(F(v)-F(u))/\ep}\,dv\,du
            \le \int_t^x\int_u^x e^{(F(x)-F(\wt x_0))/\ep}\,dv\,du
            \le |u_0|^2e^{(F(x)-F(\wt x_0))/\ep}.
    \]
    Thus,
    \begin{align*}
        \int_{-\infty}^x\int_u^x e^{(F(v)-F(u))/\ep}\,dv\,du
            &\le C_3\ep + |u_0|^2 + |u_0|^2e^{(F(x)-F(\wt x_0))/\ep}\\
        &\le C_4e^{(F(x)-F(\wt x_0))/\ep},
    \end{align*}
    where $C_4=(C_3\ep_0+|u_0|^2)e^{-(F(x)-F(\wt x_0))/\ep_0}+|u_0|^2$. Finally,
    combining this with \eqref{stepestmid}, we obtain \eqref{stepest}, where
    $C=C_1C_4$.
\end{proof}

\subsection{Behavior near the nodal point}

From this point forward, for definiteness, we assume $F(\wt x_0)<F(\wt x_1)$, so
that $x_0=\wt x_0$ and $x_1=\wt x_1$.

Having shown that $\eta$ is asymptotically flat near the minima, we would now
like to show that it behaves, weakly, like a simple function that is constant on
the domains of attraction defined in \eqref{domatt}. That is, we want to show
that $\int_{D_0}\eta\,d\vpi\sim\eta (x_0)\vpi(D_0)$ and
$\int_{D_1}\eta\,d\vpi\sim\eta(x_1)\vpi(D_1)$. (Note that we cannot use Theorem
\ref{T:Lapmeth} since $\eta$ depends on $\ep$.) Combined with
$\int\eta\,d\vpi=0$, this would give us the relative magnitudes of $\eta(x_0)$
and $\eta(x_1)$. By Theorem \ref{T:stepest}, this is equivalent to understanding
the relative magnitudes of $\eta(-\infty)$ and $\eta(\infty)$, respectively.

\begin{lemma}\label{L:bulkint}
    Choose $\de\in(0,x_1)$ such that $\xi-\de\in(x_0,0)$. Let $k$ be a positive
    integer and let $g:\bR\to\bR$ be bounded. If $g$ is continuous at $x_0$ and
    $x_1$, then
    \begin{equation}\label{bulkint1}
        \int_{-\infty}^{\xi -\de} g(x)|\eta(x)|^k e^{-F(x)/\ep}\,dx
            \sim g(x_0)|\eta (-\infty)|^k\sqrt{
                \frac{2\pi\ep}{F''(x_0)}}\,e^{-F(x_0)/\ep
            },
    \end{equation}
    and
    \begin{equation}\label{bulkint2}
        \int_\de^\infty g(x)|\eta (x)|^k e^{-F(x)/\ep}\,dx
            \sim g(x_1)|\eta(\infty)|^k
                \sqrt{\frac{2\pi\ep}{F''(x_1)}}\,e^{-F(x_1)/\ep},
    \end{equation}
    as $\ep\to0$.
\end{lemma}

\begin{proof}
    By writing $g=g^{+}-g^{-}$, $g^{+}$ and $g^{-}$ nonnegative, we may assume
    without loss of generality that $g$ is nonnegative. By Corollary
    \ref{C:nodeloc} and the fact that $\eta$ is increasing, we have that, for
    $\ep$ sufficiently small, $|\eta(x)|\le|\eta(-\infty)|$ for all
    $x\in(-\infty,\xi -\de)$. Thus,
    \[
        \int_{-\infty}^{\xi -\de} g(x)|\eta (x)|^k e^{-F(x)/\ep}\,dx
            \le |\eta(-\infty)|^k
                \int_{-\infty}^{\xi -\de} g(x)e^{-F(x)/\ep}\,dx.
    \]
    Similarly,
    \[
        \int_{-\infty}^{\xi -\de} g(x)|\eta (x)|^k e^{-F(x)/\ep}\,dx
            \ge |\eta(\xi -\de)|^k
                \int_{-\infty}^{\xi -\de} g(x)e^{-F(x)/\ep}\,dx.
    \]
    Hence, by Theorem \ref{T:stepest},
    \[
        \int_{-\infty}^{\xi -\de} g(x)|\eta (x)|^k e^{-F(x)/\ep}\,dx
            \sim |\eta(-\infty)|^k
                \int_{-\infty}^{\xi -\de} g(x)e^{-F(x)/\ep}\,dx.
    \]
    By Theorem \ref{T:Lapmeth}, this proves \eqref{bulkint1}. Replacing $F$ with
    $x\mapsto F(-x)$, Theorem \ref{T:stepest} shows that
    $\eta(\de)\sim\eta(\infty)$. Thus, the same argument can be used to obtain
    \eqref{bulkint2}.
\end{proof}

\begin{lemma}\label{L:midint}
    There exists $\de_0>0$ such that for all $\de\in(0,\de_0)$,
    \[
        \int_{\xi -\de}^\de \eta(x)e^{-F(x)/\ep}\,dx
            = o\left({
                \int_{-\infty}^{\xi -\de} |\eta(x)|e^{-F(x)/\ep}\,dx
                + \int_\de^\infty |\eta(x)|e^{-F(x)/\ep}\,dx
            }\right),
    \]
    as $\ep\to0$.
\end{lemma}

\begin{proof}
    Without loss of generality, we may assume $F(x_0)=0$. Let
    $\ga=(F(0)-F(x_1))/4>0$.  By the continuity of $F$, we may choose $\de_0>0$
    such that $F(-\de_0)>F(x_1)$ and
    \begin{equation}\label{starrev1}
        F(-\de/2) - F(x_0 - \de) - F(x_1) > 2\ga,
    \end{equation}
    for all $\de\in(0,\de_0)$.

    Let $\de\in(0,\de_0)$ be arbitrary. By Theorem \ref{T:dirshape} applied to
    $x\mapsto F(-x)$, there exists $\de'>0$ and $0<\ep_0<x_1$ such that
    \[
        |\eta(x)| \le (1 + \de')|\eta(x_1 - \ep)|P^x(\tau^X_{x_1-\ep}
            < \tau^X_{r_\ep}),
    \]
    for all $x\in(\xi -\de,-\de)\cap(r_\ep,\infty)$ and all $\ep\in(0,\ep_0)$.
    For any such $x$ and $\ep$, since $X$ is continuous and
    \[
        x_0 - \de < r_\ep < x < -\de/2 < x_1 - \ep,
    \]
    it follows that on $\{\tau^X_{x_1-\ep}<\tau^X_{r_\ep}\}$, we have
    $\tau^X_{-\de/2}<\tau^X_{x_0-\de}$, $P^x$-a.s. Hence,
    \[
        |\eta(x)| \le (1 + \de')|\eta (x_1 - \ep)|
            P^x(\tau^X_{-\de/2}<\tau^X_{x_0-\de}).
    \]
    By making $\ep_0$ smaller, if necessary, and using Theorem \ref{T:stepest}
    applied to $x\mapsto F(-x)$, this gives
    \[
        |\eta(x)| \le (1 + \de')^2|\eta(\infty)|
            P^x(\tau^X_{-\de/2}<\tau^X_{x_0-\de}),
    \]
    for all $x\in(\xi -\de,-\de)\cap(r_\ep,\infty)$ and all $\ep\in(0,\ep_0)$.
    By \eqref{starrev1}, we may apply Lemma \ref{L:whichside}, so that by making
    $\ep_0$ smaller, if necessary, we obtain
    \begin{equation}\label{HMScompare}
        |\eta(x)| \le (1 + \de')^2|\eta(\infty)|\exp\left({
                -\frac1\ep(F(-\de/2) - F(x_0 - \de) - 2\ga),
            }\right)
    \end{equation}
    for all $x\in(\xi-\de,-\de)\cap(r_\ep,\infty)$ and all $\ep\in(0,\ep_0)$. By
    \eqref{starrev1}, for fixed $x\in(\xi -\de,-\de)\cap(r_\ep,\infty)$ and
    $\ep\in(0,\ep_0)$, we may write
    \[
        |\eta(x)| \le (1 + \de')^2|\eta(\infty)|e^{-F(x_1)/\ep}.
    \]
    For fixed $x\in(-\infty,r_\ep]$, by the monotonicity of $\eta$, we have
    $|\eta(x)|\le|\eta(-\infty)|$. Therefore, for all $x\in(\xi -\de,-\de)$ and
    all $\ep\in(0,\ep_0)$, we have
    \[
        |\eta(x)| \le (1 + \de')^2(|\eta(-\infty)|
            + |\eta(\infty)|e^{-F(x_1)/\ep}).
    \]
    By Proposition \ref{P:Lapmeth}, after making $\ep_0$ smaller, if necessary,
    we have
    \begin{align}
        \int_{\xi -\de}^{-\de} |\eta(x)|e^{-F(x)/\ep}\,dx
            &\le (1 + \de')^2(|\eta(-\infty)|
            + |\eta(\infty)|e^{-F(x_1)/\ep})
            \int_{\xi -\de}^{-\de} e^{-F(x)/\ep}\,dx\notag\\
        &\le (1 + \de')^3(|\eta(-\infty)|
            + |\eta(\infty)|e^{-F(x_1)/\ep})\frac\ep{F'(\xi -\de)}.
            \label{midint1}
    \end{align}
    Let $m=\min\{F'(\xi -\de),F'(-\de),|F'(\de)|\}$. Choose $c\in \{-\de,\de\}$
    such that $F(c)=F(-\de)\wedge F(\de)$. By Proposition \ref{P:Lapmeth}, by
    making $\ep_0$ smaller, if necessary, we also have
    \begin{align}
        \int_{-\de}^\de |\eta(x)|e^{-F(x)/\ep}\,dx
            &\le (|\eta(-\infty)|
            + |\eta(\infty)|)\left({
                \int_{-\de}^0 e^{-F(x)/\ep}\,dx + \int_0^{\de} e^{-F(x)/\ep}\,dx
            }\right)\notag\\
        &\le (1 + \de')(|\eta(-\infty)|
            + |\eta(\infty)|)\frac{2\ep}m e^{-F(c)/\ep}\notag\\
        &\le (1 + \de')(|\eta(-\infty)|
            + |\eta(\infty)|e^{-F(x_1)/\ep})\frac{2\ep}m\label{midint2}
    \end{align}
    Combining \eqref{midint1} and \eqref{midint2} gives
    \begin{equation}\label{midint3}
        \int_{\xi -\de}^\de |\eta(x)|e^{-F(x)/\ep}\,dx
            \le (1 + \de')^3(|\eta(-\infty)|
            + |\eta(\infty)|e^{-F(x_1)/\ep})\frac{3\ep}m.
    \end{equation}
    Using Lemma \ref{L:bulkint}, again making $\ep_0$ smaller, if necessary, we
    have 
    \begin{align*}
        \int_{\xi-\de}^{\de} &|\eta(x)|e^{-F(x)/\ep}\,dx\\
        &\le (1 + \de')^4\left({
                \sqrt{\frac{F''(x_0)}{2\pi\ep}}
                    \int_{-\infty}^{\xi-\de}|\eta(x)|e^{-F(x)/\ep}\,dx
            }\right.\\
        &\qquad\qquad\qquad\qquad \left.{
            {} + \sqrt{\frac{F''(x_1)}{2\pi\ep}}
                \int_\de^\infty|\eta(x)|e^{-F(x)/\ep}\,dx
            }\right)
            \frac{3\ep}{m}\\
        &\le \frac{3\ep^{1/2}(1+\de')^4\sqrt{F''(x_0)\vee F''(x_1)}}{m}
            \left({\int_{-\infty}^{\xi-\de} |\eta(x)|e^{-F(x)/\ep}\,dx
            + \int_\de^\infty |\eta(x)|e^{-F(x)/\ep}\,dx}\right),
    \end{align*}
    which completes the proof.
\end{proof}

\begin{rmk}
    Although we have narrowed down the location of the nodal point, $r_\ep$, to
    the interval $(\xi-\de,\de)$, the work in \cite{Huisinga2004} suggests that
    the nodal point actually converges to $\xi$. Moreover, the caption to
    \cite[Fig.~3]{Huisinga2004}, states that a step function with discontinuity
    at $\xi$ is a candidate limit for $\eta$ as $\ep\to0$. However,
    \eqref{HMScompare} shows that $\eta(x)=o(\eta(\infty))$ for all $x<0$. In
    fact, together with Theorem \ref{T:stepest} applied to $x\mapsto F(-x)$, it
    follows that $\eta/\eta(\infty)\to1_{(0,\infty)}$, pointwise on
    $\bR\setminus\{0\}$.
\end{rmk}

\begin{prop}\label{P:relinf}
    We have
    \[
        \frac{\eta(\infty)}{|\eta(-\infty)|}
            \sim \sqrt{\frac{F''(x_1)}{F''(x_0)}}\,e^{(F(x_1)-F(x_0))/\ep},
    \]
    as $\ep\to 0$.
\end{prop}

\begin{proof}
    Choose $\de$ such that Lemma \ref{L:bulkint} and Lemma \ref{L:midint} hold.
    Let
    \begin{align*}
        \ka_{1,\ep} &= \int_{-\infty}^{\xi-\de} \eta(x)e^{-F(x)/\ep}\,dx
            = -\int_{-\infty}^{\xi-\de} |\eta (x)|e^{-F(x)/\ep}\,dx,\\
        \ka_{2,\ep} &= \int_\de^\infty \eta(x)e^{-F(x)/\ep}\,dx
            =\int_\de^\infty |\eta(x)|e^{-F(x)/\ep}\,dx,\\
        \ka_{3,\ep} &= \int_{\xi-\de}^\de \eta(x)e^{-F(x)/\ep}\,dx.
    \end{align*}
    Since $\int_\bR \eta(x)e^{-F(x)/\ep}\,dx=0$, we have that
    $|\ka_{1,\ep}|=|\ka_{2,\ep}|+\ka_{3,\ep}$. By Lemma \ref{L:midint}, we also
    have that $\ka_{3,\ep}=o(|\ka_{1,\ep}|+|\ka_{2,\ep}|)$.

    Since $\ka_{3,\ep}=o(|\ka_{1,\ep}|+|\ka_{2,\ep}|)$, there exists $\ep_0>0$
    such that $|\ka_{1,\ep}|+|\ka_{2,\ep}|>0$ and
    \[
        \frac{|\ka_{3,\ep}|}{|\ka_{1,\ep}| + |\ka_{2,\ep}|} < 1,
    \]
    for all $\ep\in(0,\ep_0)$. Hence, for any such $\ep$, we may write
    \[
        \frac{2\left({
                \frac {\ka_{3,\ep}}{|\ka_{1,\ep}| + |\ka_{2,\ep}|}
            }\right)}
            {1 - \left({
                \frac{\ka_{3,\ep}}{|\ka_{1,\ep}| + |\ka_{2,\ep}|}
            }\right)}
            = \frac{2\ka_{3,\ep}}{|\ka_{1,\ep}| + |\ka_{2,\ep}| - \ka_{3,\ep}}
            = \frac{\ka_{3,\ep}}{|\ka_{2,\ep}|},
    \]
    which implies $|\ka_{2,\ep}|>0$ for all such $\ep$, and also shows that
    $\ka_{3, \ep}/|\ka_{2,\ep}| \to 0$ as $\ep\to0$. Therefore,
    $|\ka_{1,\ep}|/|\ka_{2,\ep}| = 1 + \ka_{3,\ep}/|\ka_{2,\ep}| \to 1$ as
    $\ep\to0$. That is, $|\ka_{1,\ep}| \sim|\ka_{2,\ep}|$. Applying Lemma
    \ref{L:bulkint} finishes the proof.
\end{proof}

In the following theorem, we improve the results of Lemma \ref{L:bulkint} in the
case $k=1$, to extend the intervals of integration to include the entire domains
of attraction.

\begin{thm}\label{T:bulkintII}
    If $g\in L^\infty(\bR)$ is continuous at $x_0$ and $x_1$, then
    \begin{equation}\label{bulkintII1}
        \int_{-\infty}^0 g(x)\eta(x)e^{-F(x)/\ep}\,dx
            \sim g(x_0)\eta(-\infty)\sqrt{
                    \frac{2\pi\ep}{F''(x_0)}
                }\,e^{-F(x_0)/\ep},
    \end{equation}
    and
    \begin{equation}\label{bulkintII2}
        \int_0^\infty g(x)\eta(x)e^{-F(x)/\ep}\,dx
            \sim g(x_1)\eta(\infty)\sqrt{
                    \frac{2\pi\ep}{F''(x_1)}
                }\,e^{-F(x_1)/\ep},
    \end{equation}
    as $\ep\to0$, provided the integrals exist for sufficiently small $\ep$.
    Consequently,
    \begin{equation}\label{wt-eta-int}
        \int g\eta\,d\vpi \sim (g(x_0) - g(x_1))\eta(-\infty),
    \end{equation}
    as $\ep\to0$.
\end{thm}

\begin{proof}
    Without loss of generality, we may assume $F(x_0)=0$. Choose $\de$ so that
    Lemma \ref{L:midint} applies. By \eqref{midint3} and Proposition
    \ref{P:relinf},
    \[
        \left|{\int_{\xi-\de}^0 g(x)\eta(x)e^{-F(x)/\ep}\,dx}\right|
            \le \|g\|_\infty(1 + \de')^4\left({
                1 + \sqrt{\frac{F''(x_1)}{F''(x_0)}}
            }\right)|\eta(-\infty)|\frac{6\ep}m,
    \]
    for $\ep$ sufficiently small, where
    $m=\min\{F'(\xi-\de),F'(-\de),|F'(\de)|\}$. Thus, to prove
    \eqref{bulkintII1}, it suffices to show that
    \[
        \int_{-\infty}^{\xi-\de} g(x)\eta(x)e^{-F(x)/\ep}\,dx
            \sim g(x_0)\eta(-\infty)\sqrt{\frac{2\pi\ep}{F''(x_0)}}.
    \]
    But this follows from \eqref{bulkint1} with $k=1$ and the fact that $\eta<0$
    on $(-\infty,\xi-\de)$.

    Using Proposition \ref{P:relinf}, to prove \eqref{bulkintII2}, it suffices
    to show that
    \[
        \int_0^\infty g(x)\eta(x)e^{-F(x)/\ep}\,dx
            \sim -g(x_1)\eta(-\infty)\sqrt{\frac{2\pi\ep}{F''(x_0)}}.
    \]
    As above, by \eqref{midint3} and Proposition \ref{P:relinf}, it suffices to
    show that
    \[
        \int_\de^\infty g(x)\eta(x)e^{-F(x)/\ep}\,dx
            \sim -g(x_1)\eta(-\infty)\sqrt{\frac{2\pi\ep}{F''(x_0)}}.
    \]
    But this follows from \eqref{bulkint2}, Proposition \ref{P:relinf}, and the
    fact that $\eta >0$ on $(\de,\infty)$. Finally, combining these results with
    Proposition \ref{P:relinf} and Theorem \ref{T:Lapmeth}, we obtain
    \[
        \eta(-\infty)^{-1}\int g\eta\,d\vpi \to g(x_0) - g(x_1),
    \]
    as $\ep\to0$.
\end{proof}

\section{Asymptotic behavior of the coupled process}\label{S:asympcoup}

Recall that we are assuming $F$ is a double-well potential in one dimension,
with $x_0<0<x_1$ and $F(x_0)<F(x_1)$. Here, the $x_j$'s are the local minima and
$0$ is the local maximum.  

Our construction of the coupling is dependent on our choice of $Q\in\bR^{2\times
2}$ and $\xi =\xi^{(1)}$ in the coupling construction outlined in the
introduction (see \cite[Theorem 3.8]{Kurtz2020} for more details). We begin with
a lemma that characterizes all the admissible choices for $Q$ and $\xi$.

\begin{lemma}\label{L:matcond}
    Let $\xi_0,\xi_1\in\bR$. Define $a_j=\la\xi_j/(\xi_j-\xi_{1-j})$. Then
    \[
        Q=\begin{pmatrix}
            -a_0 & a_0\\
            a_1 & -a_1
        \end{pmatrix}
    \]
    is the generator of a continuous-time Markov chain with state space
    $E_0=\{0,1\}$, eigenvalues $\{0,-\la\}$, and corresponding eigenvectors
    $(1,1)^T$ and $\xi=(\xi_1,\xi_2)^T$ satisfying $\al_j=1+\xi_j\eta>0$ if and
    only if the following conditions hold:   
    \begin{enumerate}[(i)]
        \item   $\ds{-\frac1{\eta(\infty)}\le\xi_j\le\frac1{|\eta(-\infty)|}}$,
                for $j=0,1$, and
        \item   $\xi_0\xi_1<0$.
    \end{enumerate}
\end{lemma}

\begin{proof}
    Note that the $a_j$ are defined precisely so that $Q$ has the given
    eigenvalues and eigenvectors. Also, $\al_j=1+\xi_j\eta >0$ if and only if
    (i). And the $a_j$ are both positive if and only if (ii).
\end{proof}

For any such choice of $\xi$ as in Lemma \ref{L:matcond}, we obtain a
coupled process $(X,Y)$ with generator $B$ given by \eqref{Bgen} and initial
distribution $\nu$ given by \eqref{nudef}. This process is cadlag, $X$ satisfies
the SDE given by \eqref{SDE}, $Y$ is a continuous-time Markov chain with
generator $Q$, and, by \eqref{mainRd},
\begin{equation}\label{mainspec}
    P(X(t) \in \Ga \mid Y(t) = j) = \int_\Ga \al_j(x)\,\vpi(dx)
        = \vpi(\Ga) + \xi_j\int_\Ga \eta(x)\,\vpi(dx),
\end{equation}
for $j=0,1$ and all Borel sets $\Ga\subset\bR$. Recall that
$\vpi=\mu(\bR)^{-1}\mu$ and $\mu(dx)=e^{-F(x)/\ep}\,dx$.

For each fixed $\ep>0$, we may choose a different $\xi$. Hence, all of these
objects, in fact, depend on $\ep$. We will, however, suppress that dependence in
the notation.

\begin{thm}\label{T:halfmain}
    The following are equivalent to \eqref{trackdeep}:
    \begin{enumerate}[(a)]
        \item   $\xi_0=o(|\eta(-\infty)|^{-1})$ as $\ep\to0$.
        \item   $E[g(X(t)) \mid Y(t) = 0] - E^\vpi[g(X(0)) \mid X(0) < 0] \to 0$
                as $\ep\to0$, for each $t\ge0$ and each bounded, measurable
                $g:\bR\to\bR$ that is continuous at $x_0$ and $x_1$.
    \end{enumerate}
\end{thm}

\begin{proof}
    Note that
    \begin{multline*}
        E[g(X(t)) \mid Y(t) = 0] - E^\vpi[g(X(0)) \mid X(0) < 0]\\
        = \int g(x)(1 + \xi_0\eta(x))\,\vpi(dx)
            - \vpi((-\infty,0))^{-1}\int_{-\infty}^0 g(x)\,\vpi(dx).
    \end{multline*}
    Since $\vpi((-\infty,0))^{-1}\to 1$ and $\left|{\int_0^\infty
    g\,d\vpi}\right|\le\|g\|_\infty\vpi((0,\infty))\to 0$, in order to prove
    that (a) and (b) are equivalent, it suffices to show that
    $\xi_0=o(|\eta(-\infty)|^{-1})$ if and only if $\xi_0\int g\eta\,d\vpi\to 0$
    for all $g$ satisfying the hypotheses. But this follows from 
    \eqref{wt-eta-int}.

    That (b) implies \eqref{trackdeep} is trivial. Assume \eqref{trackdeep}.
    Since
    \[
        P(X(t) < 0 \mid Y(t) = 0)
            = \vpi((-\infty,0)) + \xi_0\int_{-\infty}^0 \eta(x)\,\vpi(dx),
    \]
    and $\vpi((-\infty,0))\to1$, it follow that
    $\xi_0\int_{-\infty}^0\eta(x)\,\vpi(dx) \to 0$. By \eqref{wt-eta-int} with
    $g=1_{(-\infty,0)}$, we have $\int_{-\infty}^0\eta(x)\,\vpi(dx) \sim
    \eta(-\infty)$, and (a) follows.
\end{proof}

\begin{thm}\label{T:halfmain2}
    The following are equivalent to \eqref{trackshallow}:
    \begin{enumerate}[(a)]
        \item   $\xi_1\sim|\eta(-\infty)|^{-1}$.
        \item   $E[g(X(t)) \mid Y(t) = 1] - E^\vpi[g(X(0)) \mid X(0) > 0] \to 0$
                as $\ep\to0$, for each $t\ge0$ and each bounded, measurable
                $g:\bR\to\bR$ that is continuous at $x_0$ and $x_1$.
    \end{enumerate}
    Moreover, \eqref{trackshallow} implies \eqref{trackdeep}.
\end{thm}

\begin{proof}
    Note that
    \begin{multline*}
        E[g(X(t)) \mid Y(t) = 1] - E^\vpi[g(X(0)) \mid X(0) > 0]\\
        = \int g(x)(1 + \xi_1\eta(x))\,\vpi(dx)
            - \vpi((0,\infty))^{-1}\int_0^\infty g(x)\,\vpi(dx)\\
        = \int g\,d\vpi
            - \vpi((0,\infty))^{-1}\int_{(0,\infty)}g\,d\vpi
            + \xi_1\int g\eta\,d\vpi
    \end{multline*}
    To prove that (a) and (b) are equivalent, by \eqref{Lapmeth}, it suffices to
    show that $\xi_1\sim|\eta(-\infty)|^{-1}$ if and only if $\xi_1\int
    g\eta\,d\vpi\to -(g(x_0)-g(x_1))$ for all $g$ satisfying the hypotheses. But
    this follows from \eqref{wt-eta-int}.  

    That (b) implies \eqref{trackshallow} is trivial. Assume
    \eqref{trackshallow}. Since
    \[
        P(X(t) > 0 \mid Y(t) = 1)
            = \vpi((0,\infty)) + \xi_1\int_0^\infty \eta(x)\,\vpi(dx),
    \]
    and $\vpi((0,\infty))\to 0$, it follows that $\xi_1\int_0^\infty
    \eta(x)\,\vpi(dx)\to 1$. By \eqref{wt-eta-int} with $g=1_{(0,\infty)}$, we
    have $\int_0^\infty\eta(x)\,\vpi(dx)\sim|\eta(-\infty)|$, and (a) follows.

    Finally, assume \eqref{trackshallow}. Then (a) holds. By Lemma
    \ref{L:matcond}, we have $-\eta(\infty)^{-1}\le\xi_0<0$ for sufficiently
    small $\ep$. In particular, $|\xi_0|\le\eta(\infty)^{-1}$, so Theorem
    \ref{T:halfmain}(a) follows from Proposition \ref{P:relinf}.
\end{proof}

\begin{thm}\label{T:jumprates}
    Let $0<\rho <|x_0|\wedge x_1$. Then $\xi_0=o(\xi_1)$ if and only if
    \begin{equation}\label{t-todeep}
        E^1[\tau^Y_0] \sim E^{x_1}[\tau^X_{B_\rho(x_0)}]
            \sim \la^{-1}
            \sim \frac{2\pi}{|F''(0)F''(x_1)|^{1/2}} e^{(F(0) - F(x_1))/\ep},
    \end{equation}
    as $\ep\to 0$. And $\xi_1/\xi_0\sim\eta(\infty)/\eta(-\infty)$ if and only
    if
    \begin{equation}\label{t-fromdeep}
        E^0[\tau^Y_1] \sim E^{x_0}[\tau^X_{B_\rho(x_1)}]
            \sim \frac{2\pi}{|F''(0)F''(x_0)|^{1/2}} e^{(F(0) - F(x_0))/\ep},
    \end{equation}
    as $\ep\to0$. Moreover, \eqref{t-fromdeep} implies \eqref{t-todeep}, which
    implies \eqref{trackdeep}. Also, \eqref{trackshallow} implies
    \eqref{t-todeep}.
\end{thm}

\begin{proof}
    By \eqref{ttasym} and \eqref{evasym}, we need only determine the asymptotics
    of $a_0$ and $a_1$. Recall that $a_j=\la\xi_j/(\xi_j-\xi_{1-j})$. Thus,
    \begin{equation}\label{jumprates}
        E^j[\tau^Y_{1-j}] = a_j^{-1}
            = \la^{-1}\left({1 - \frac{\xi_{1-j}}{\xi_j}}\right)
            \sim \left({1 - \frac{\xi_{1-j}}{\xi_j}}\right)
                \frac{2\pi}{|F''(0)F''(x_1)|^{1/2}} e^{(F(0) - F(x_1))/\ep},
    \end{equation}
    so the first biconditional follows immediately. The second biconditional
    then follows from Proposition \ref{P:relinf}.

    By Proposition \ref{P:relinf}, we have \eqref{t-fromdeep} implies
    \eqref{t-todeep}. By Lemma \ref{L:matcond}, we have
    $|\xi_0/\xi_1|\ge|\xi_0\eta(-\infty)|$, so that $\xi_0=o(\xi_1)$ implies
    $\xi_0=o(|\eta(-\infty)|^{-1})$. Hence, \eqref{t-todeep} implies that
    Theorem \ref{T:halfmain}(a) holds, which is equivalent to \eqref{trackdeep}.

    Finally, suppose \eqref{trackshallow} holds. By Theorems \ref{T:halfmain}
    and \ref{T:halfmain2}, we have that $\xi_1\sim|\eta(-\infty)|^{-1}$ and
    $\xi_0=o(|\eta(-\infty)|^{-1})$, so that $\xi_0=o(\xi_1)$, which is
    equivalent to \eqref{t-todeep}.
\end{proof}

\begin{thm}\label{T:fulltrack}
    The Markov chain fully tracks the diffusion, in the sense that
    \eqref{trackdeep}-\eqref{time2shallow} all hold, if and only if
    $\xi_0\sim-\eta(\infty)^{-1}$ and $\xi_1\sim|\eta(-\infty)|^{-1}$.
\end{thm}

\begin{proof}
    Suppose \eqref{trackdeep}-\eqref{time2shallow} hold. Then, by Theorem
    \ref{T:halfmain2}, we have $\xi_1\sim|\eta(-\infty)|^{-1}$. Since
    \eqref{time2shallow} is equivalent to \eqref{t-fromdeep}, we also have, by
    Theorem \ref{T:jumprates}, that $\xi_1/\xi_0\sim\eta(\infty)/\eta(-\infty)$.
    Thus, $\xi_0\sim-\eta(\infty)^{-1}$.

    Conversely, suppose $\xi_0\sim-\eta(\infty)^{-1}$ and
    $\xi_1\sim|\eta(-\infty)|^{-1}$. Theorem \ref{T:halfmain2} gives us
    \eqref{trackshallow} and \eqref{trackdeep}. Theorem \ref{T:jumprates} gives
    us \eqref{t-fromdeep} and \eqref{t-todeep}, which are equivalent to
    \eqref{time2shallow} and \eqref{time2deep}, respectively.
\end{proof}

In this section, we have established that \eqref{trackshallow} implies
\eqref{time2deep} implies \eqref{trackdeep}, and \eqref{time2shallow} implies
\eqref{time2deep}. Example \ref{1steg} shows that it is possible to have all
four conditions holding. The remaining examples illustrate that there are no
implications besides those already mentioned.

\begin{expl}\label{1steg}
    Let $\xi_0=-f(\ep)\eta(\infty)^{-1}$ and $\xi_1=g(\ep)|\eta(-\infty)|^{-1}$,
    where $0<f,g\le 1$ with $f,g\to 1$ as $\ep\to 0$. By Lemma \ref{L:matcond}
    and Theorem \ref{T:fulltrack}, this is the most general family of choices
    such that the resulting coupling sequence satisfies
    \eqref{trackdeep}-\eqref{time2shallow}.  
\end{expl}

In the remaining examples, let
\[
    L(\ep) = \sqrt{\frac{F''(x_1)}{F''(x_0)}}\,e^{-(F(x_1) - F(x_0))/\ep},
\]
so that by Proposition \ref{P:relinf}, we have
$\eta(\infty)/\eta(-\infty)\sim-L(\ep)^{-1}$. Choose $0<f\le 1$ and $h\ge L$,
and let $g=L/h$, so that $0<g\le 1$. Let $\xi_0=-f(\ep)\eta(\infty)^{-1}$ and
$\xi_1=g(\ep)|\eta(-\infty)|^{-1}$. By Lemma \ref{L:matcond}, these are
admissible choices for $\xi_0$ and $\xi_1$.

Note that $\xi_0\sim-f(\ep)L(\ep)|\eta(-\infty)|^{-1}$, so that by Theorem
\ref{T:halfmain}, we have \eqref{trackdeep} in all these examples. Also note
that by Theorem \ref{T:halfmain2}, we have \eqref{trackshallow} if and only if
$h\sim L$. For applying Theorem \ref{T:jumprates}, note that
$\xi_1/\xi_0\sim-g/(fL)=-1/(fh)$. Thus, \eqref{time2deep} holds if and only if
$fh\to 0$ and \eqref{time2shallow} holds if and only $fh\sim L$.

\begin{expl}
    Let $f=h=1$. Then none of \eqref{trackshallow}, \eqref{time2deep}, or
    \eqref{time2shallow} hold, so we see that \eqref{trackdeep} does not imply
    any of the other conditions.
\end{expl}

\begin{expl}
    Let $f=1$ and $h=\sqrt {L}$. In this case, we have \eqref{time2deep}, but
    neither \eqref{trackshallow} nor \eqref{time2shallow} hold. Hence,
    \eqref{time2deep} implies neither \eqref{trackshallow} nor
    \eqref{time2shallow}.
\end{expl}

\begin{expl}
    Let $f=h=L$. In this case, \eqref{trackshallow} and \eqref{time2deep} hold,
    but \eqref{time2shallow} does not, showing that \eqref{trackshallow} does
    not imply \eqref{time2shallow}.
\end{expl}

\begin{expl}
    Let $f=h=\sqrt L$. Here we have \eqref{time2deep} and \eqref{time2shallow},
    but not \eqref{trackshallow}, showing that \eqref{time2shallow} does not
    imply \eqref{trackshallow}.
\end{expl}

\section*{Acknowledgments}

This paper was completed while the first author was visiting the University of
California, San Diego with the support of the Charles Lee Powell Foundation. The
hospitality of that institution, particularly that of Professor Ruth Williams,
was greatly appreciated.

\renewcommand{\theequation}{A.\arabic{equation}}
\appendix
\setcounter{equation}{0}

\section{Appendix}

\subsection{Results of Freidlin and Wentzell}

Let $b:\bR^d\to\bR^d$ be Lipschitz and let $\ph_{x,b}$ be the unique solution to
$\ph'_{x,b}=b(\ph_{x,b})$ with $\ph_{x,b}(0)=x$. For $\ep>0$, let $X_{\ep,b}$ be
defined by
\[
    X_{\ep,b}(t)
        = X_{\ep,b}(0) + \int_0^t b(X_{\ep,b}(s))\,ds + \sqrt{2\ep}W(t),
\]
where $W$ is a standard $d$-dimensional Brownian motion. As in Section
\ref{S:intro}, if $F:\bR^d\to\bR$ is given, then $\ph_x=\ph_{x,-\nab F}$ and
$X_\ep = X_{\ep,-\nab F}$. For the $F$ we use later, $-\nab F$ is \textit{not}
Lipschitz. This will cause no difficulty, however, since it will be locally
Lipschitz, and we will only apply these theorems on compact sets.

This first theorem is \cite[Theorem 2.40]{Olivieri2005}. It describes the
asymptotic mean time to leave a domain of attraction.

\begin{thm}\label{T:FandW} 
    Let $F:\bR^d\to\bR$ have continuous and bounded derivatives up to second
    order. Let $D$ be a bounded open domain in $\bR^d$ with boundary $\pa D$ of
    class $C^2$ and $\ang{-\nab F(x),n(x)} <0$ for all $x\in\pa D$, where $n(x)$
    is the outward unit normal vector to $\pa D$ at $x$.

    Let $x_0\in D$. Assume that if $G$ is a neighborhood of $x_0$, then there
    exists a neighborhood $\wt G$ of $x_0$ such that $\wt G\subset G$ and, for
    all $x\in\wt G$, we have $\ph_x([0,\infty))\subset G$ and $\ph_x(t)\to x_0$
    as $t\to\infty$.  Further assume that, for each $x\in\ol D$, we have
    $\ph_x((0,\infty))\subset D$.

    Then for any $x\in D$,
    \begin{enumerate}[(i)] 
        \item   $\ds{
                    \lim_{\ep\to 0}2\ep\log E^x[\tau(D^c)]
                        = \inf_{y\in\pa D}2(F(y)-F(x_0)) =: V_0
                }$, and  
        \item   for all $\ze>0$, we have $\ds{
                    \lim_{\ep\to 0} P^x(e^{(V_0-\ze)/(2\ep)}
                        < \tau (D^c)<e^{(V_0+\ze)/(2\ep)}) = 1
                }$.
    \end{enumerate}
    Moreover, both convergences hold uniformly in $x$ on each compact subset of
    $D$.  
\end{thm}

This next theorem is \cite[Lemma 2.34(b)]{Olivieri2005}. It asserts that the
diffusion cannot linger for long inside the domain of attraction without quickly
coming into a small neighborhood of the associated minimum.

\begin{thm}\label{T:cantdance} 
    Assume the hypotheses of Theorem \ref{T:FandW}. Fix $\de>0$. Then there
    exists $C>0$, $T>0$, and $\ep_0>0$ such that
    \[
        P^x(\tau(D^c \cup B_\de(x_0)) > t) \le e^{-C(t - T)/(2\ep)},
    \]
    for all $x\in\ol D\setminus B_\de(x_0)$, all $t>T$, and all $\ep<\ep_0$.
\end{thm}

The last result we need gives the probability of leaving the domain of
attraction through a given point. To state this result, we need some preliminary
notation and definitions. See \cite[Section 5.3]{Olivieri2005} for more details.

Let $u:[0,T]\to\bR^d$. If $u$ is absolutely continuous, define
\[
    I_T(u) = \frac12\int_0^T |u'(s) - b(u(s))|^2\,ds,
\]
and define $I_T(u)=\infty$ otherwise.

Let $G$ be a bounded domain in $\bR^d$ with $\pa G$ of class $C^2$ and define
\begin{align*}
    V(x,y) &= \inf\{I_T(u) \mid T > 0,
        u:[0,T] \to \bR^d,
        u(0) = x, u(T) = y\}\\
    V_G(x,y) &= \inf\{I_T(u) \mid T > 0,
        u:[0,T] \to G \cup \pa G,
        u(0) = x, u(T) = y\}.
\end{align*}
The functions $V$ and $V_G$ are continuous on $\bR^d\times\bR^d$ and $(G\cup\pa
G)\times (G\cup\pa G)$, respectively. We have $V_G(x,y)\ge V(x,y)$ for all
$x,y\in G\cup\pa G$. Also, for all $x,y\in G$, if $V_G(x,y)\le\min_{z\in\pa
G}V(x,z)$, then $V_G(x,y)=V(x,y)$.

Note that if $\ph_{x,b}(t)=y$ for some $t>0$ and $\ph_{x,b}([0,t])\subset
G\cup\pa G$, then $V_G(x,y)=0$. An equivalence relation on $G\cup\pa G$ is
defined by $x\sim_Gy$ if and only if $V_G(x,y)=V_G(y,x)=0$. It can be shown that
if the equivalence class of $y$ is nontrivial, then $\ph_{y,b}([0,\infty))$ is
contained in that equivalence class.

The $\om$-limit set of a point $y\in\bR^d$ is denoted by $\om(y)$ and defined as
the set of accumulation points of $\ph_{y,b}([0,\infty))$. Assume that $G$
contains a finite number of compact sets $K_1,\ldots,K_{\ell}$ such that each
$K_i$ is an equivalence class of $\sim_G$. Assume further that, for all
$y\in\bR^d$, if $\om(y)\subset G\cup\pa G$, then $\om(y)\subset K_i$ for some
$i$.

The function $V_G$ is constant on $K_i\times K_j$, so we let $V_G(K_i,K_j)$,
$V_G(x,K_i)$, and $V_G(K_i,x)$ denote this common value. Also, $V_G(K_i,\pa
G)=\inf_{y\in\pa G}V_G(K_i,y)$.  

Given a finite set $\cL$ and a nonempty, proper subset $\cQ\subset\cL$, let
$\bG(\cQ)$ denote the set of directed graphs on $\cL$ with arrows $i\to j$,
$i\in\cL\setminus\cQ$, $j\in\cL$, $j\ne i$, such that: (i) from each
$i\in\cL\setminus\cQ$ exactly one arrow is issued; (ii) for each
$i\in\cL\setminus\cQ$ there is a chain of arrows starting at $i$ and finishing
at some point in $\cQ$. If $j$ is such a point we say that the graph leads $i$
to $j$. For $i\in\cL\setminus\cQ$ and $j\in\cQ$, the set of graphs in $\bG(\cQ)$
leading $i$ to $j$ is denoted by $\bG_{i,j}(\cQ)$.

With $\cL=\{K_1,\ldots,K_\ell,\pa G\}$, let
\[
    M_G = \min_{g\in\bG(\pa G)}\sum_{(\al\to\be)\in g} V_G(\al,\be).
\]
If $x\in G$ and $y\in\pa G$, then with $\cL=\{K_1,\ldots,K_\ell,x,y,\pa G\}$,
let
\[
    M_G(x,y)
        = \min_{g\in\bG_{x,y}(\{y,\pa G\})}\sum_{(\al\to\be)\in g} V_G(\al,\be).
\]
The following theorem is \cite[Theorem 5.19]{Olivieri2005}.

\begin{thm}\label{T:FandW2}
Under the above assumptions and notation, for any compact set $K\subset G$,
$\ga>0$, and $\de>0$, there exists $\ep_0>0$ and $\de_0\in(0,\de)$ so that for
any $x\in K$, $y\in\pa G$, and $\ep\in(0,\ep_0)$, we have
\[
    \exp\left({-\frac{M_G(x,y) - M_G + 2\ga}{2\ep}}\right)
        \le P^x(X_{\ep,b}(\tau) \in B_{\de_0}(y))
        \le\exp\left({-\frac{M_G(x,y) - M_G - 2\ga}{2\ep}}\right),
\]
where $\tau=\tau^{X_{\ep,b}}(\bR^d\setminus G)$.
\end{thm}

The next two results are auxiliary results which are needed to apply Theorem
\ref{T:FandW2}. The first is \cite[Proposition 2.37]{Olivieri2005}.

\begin{prop}\label{P:FandW}
    Under the assumptions of Theorem \ref{T:FandW}, we have
    \[
        V(x_0,y) = 2(F(y) - F(x_0)),
    \]
    for all $y\in\ol D$.
\end{prop}

\begin{lemma}\label{L:FandW} 
    Let $b=-\nab F$, where $F$ is as in Theorem \ref{T:FandW}. If there exists
    $T_0>0$ such that $\ph_x(T_0)=y$, then $V(x,y)=0$ and
    $V(y,x)=2(F(x)-F(y))$.
\end{lemma}

\begin{proof}
    Since $\ph_x'=b(\ph_x)$, we have $I_{T_0}(\ph_x)=0$, which implies
    $V(x,y)=0$. Let $T>0$ and let $\ph:[0,T]\to\bR^d$ satisfy $\ph(0)=y$ and
    $\ph(T)=x$. Then
    \begin{align*}
        I_T(\ph) &= \frac12\int_0^T |\ph'(s) - b(\ph(s))|^2\,ds\\
        &= \frac12\int_0^T |\ph'(s) + b(\ph(s))|^2\,ds
            - 2\int_0^T \ang{\ph'(s),b(\ph(s))}\,ds\\
        &= \frac12\int_0^T |\ph'(s) + b(\ph(s))|^2\,ds
            + 2\int_0^T \ang{\ph'(s),\nab F(\ph(s))}\,ds\\
        &= \frac12\int_0^T |\ph'(s) + b(\ph(s))|^2\,ds + 2(F(x) - F(y)).
    \end{align*}
    This shows that $V(y,x)\ge 2(F(x) - F(y))$. Now let $\psi(t)=\ph_x(T_0-t)$.
    Then $\psi (0)=y$, $\psi(T_0)=x$, and $\psi'=-b(\psi )$. Hence, $V(y,x)\le
    I_{T_0}(\psi)=2(F(x)-F(y))$.
\end{proof}

\subsection{The Laplace method}

Finally, we need two classical results of Laplace that allow us to estimate
exponential integrals. The following two results can be found in
\cite[pp.~36--37]{Erdelyi1956}. The notation $a\sim b$ means that $a/b\to 1$.

\begin{thm}\label{T:Lapmeth}
    Let $I\subset\bR$ be a (possibly infinite) open interval, $F\in C^2(I)$, and
    $x_0\in I$. Suppose $g$ is continuous at $x_0$. If $F(x_0)$ is the unique
    global minimum of $F$ on $I$, and $F''(x_0)>0$, then
    \begin{equation}\label{Lapmeth}
        \int_I g(x)e^{-F(x)/\ep}\,dx
            \sim g(x_0)\sqrt{\frac{2\pi\ep}{F''(x_0)}}\,e^{-F(x_0)/\ep},
    \end{equation}
    as $\ep\to 0$, provided the left-hand side exists for sufficiently small
    $\ep$.
\end{thm}

\begin{prop}\label{P:Lapmeth} 
    Let $-\infty <a<x_0<b\le\infty$ and $F\in C^1(a,b)$. Suppose $g$ is
    continuous at $x_0$. If $F(x_0)$ is the unique global minimum of $F$ on
    $[x_0,b)$ and $F'(x_0)>0$, then
    \begin{equation}\label{Lapmeth2}
        \int_{x_0}^b g(x)e^{-F(x)/\ep}\,dx
            \sim g(x_0)\frac\ep{F'(x_0)}\,e^{-F(x_0)/\ep},
    \end{equation}
    as $\ep\to 0$, provided the left-hand side exists for sufficiently small
    $\ep$.
\end{prop}


\begin{thebibliography}{10}

\bibitem{BEGK}
Anton Bovier, Michael Eckhoff, V{\'e}ronique Gayrard, and Markus Klein.
\newblock Metastability in reversible diffusion processes. {I}. {S}harp
  asymptotics for capacities and exit times.
\newblock {\em J. Eur. Math. Soc. (JEMS)}, 6(4):399--424, 2004.

\bibitem{BGK}
Anton Bovier, V{\'e}ronique Gayrard, and Markus Klein.
\newblock Metastability in reversible diffusion processes. {II}. {P}recise
  asymptotics for small eigenvalues.
\newblock {\em J. Eur. Math. Soc. (JEMS)}, 7(1):69--99, 2005.

\bibitem{Courant1953}
R.~Courant and D.~Hilbert.
\newblock {\em Methods of mathematical physics. {V}ol. {I}}.
\newblock Interscience Publishers, Inc., New York, N.Y., 1953.

\bibitem{E}
Michael Eckhoff.
\newblock Precise asymptotics of small eigenvalues of reversible diffusions in
  the metastable regime.
\newblock {\em Ann. Probab.}, 33(1):244--299, 2005.

\bibitem{Erdelyi1956}
A.~Erd\'elyi.
\newblock {\em Asymptotic expansions}.
\newblock Dover Publications, Inc., New York, 1956.

\bibitem{EK}
Stewart~N. Ethier and Thomas~G. Kurtz.
\newblock {\em Markov processes}.
\newblock Wiley Series in Probability and Mathematical Statistics: Probability
  and Mathematical Statistics. John Wiley \& Sons Inc., New York, 1986.
\newblock Characterization and convergence.

\bibitem{FW}
M.~I. Freidlin and A.~D. Wentzell.
\newblock {\em Random perturbations of dynamical systems}, volume 260 of {\em
  Grundlehren der Mathematischen Wissenschaften [Fundamental Principles of
  Mathematical Sciences]}.
\newblock Springer-Verlag, New York, second edition, 1998.
\newblock Translated from the 1979 Russian original by Joseph Sz\"ucs.

\bibitem{Huisinga2004}
Wilhelm Huisinga, Sean Meyn, and Christof Sch{\"u}tte.
\newblock Phase transitions and metastability in {M}arkovian and molecular
  systems.
\newblock {\em Ann. Appl. Probab.}, 14(1):419--458, 2004.

\bibitem{Karatzas1991}
Ioannis Karatzas and Steven~E. Shreve.
\newblock {\em Brownian motion and stochastic calculus}, volume 113 of {\em
  Graduate Texts in Mathematics}.
\newblock Springer-Verlag, New York, second edition, 1991.

\bibitem{Kurtz2020}
Thomas~G. Kurtz and Jason Swanson.
\newblock Finite {M}arkov chains coupled to general {M}arkov processes and an
  application to metastability {I}.
\newblock Preprint, 2020.

\bibitem{Olivieri2005}
Enzo Olivieri and Maria~Eul\'alia Vares.
\newblock {\em Large deviations and metastability}, volume 100 of {\em
  Encyclopedia of Mathematics and its Applications}.
\newblock Cambridge University Press, Cambridge, 2005.

\bibitem{Sugiura1995}
Makoto Sugiura.
\newblock Metastable behaviors of diffusion processes with small parameter.
\newblock {\em J. Math. Soc. Japan}, 47(4):755--788, 1995.

\bibitem{Sugiura2001}
Makoto Sugiura.
\newblock Asymptotic behaviors on the small parameter exit problems and the
  singularly perturbation problems.
\newblock {\em Ryukyu Math. J.}, 14:79--118, 2001.

\end{thebibliography}
\end{document}